\documentclass[12pt]{amsart}

\usepackage{amsfonts,graphics,amsmath,amsthm,amsfonts,amscd,amssymb,amsmath,latexsym,multicol,mathrsfs}
\usepackage{epsfig,url}
\usepackage{flafter}
\usepackage{fancyhdr}
\usepackage{color}
\usepackage{hyperref}
\usepackage[all]{xy}
\hypersetup{colorlinks=true, linkcolor=black}

\numberwithin{equation}{section}

\newtheorem{corollary}[subsection]{Corollary}
\newtheorem{lemma}[subsection]{Lemma}
\newtheorem{proposition}[subsection]{Proposition}
\newtheorem{theorem}[subsection]{Theorem}

\theoremstyle{definition}
\newtheorem{definition}[subsection]{Definition}

\newtheorem{example}[subsection]{Example}
\newtheorem{remark}[subsection]{Remark}

\newcommand{\mumu}{\boldsymbol{\mu}}
\newcommand{\WD}{\mathrm{W}(\mathrm{D}_5)}
\newcommand{\WE}{\mathrm{W}(\mathrm{E}_6)}
\newcommand{\PP}{\mathbb P}

\newcommand{\DD}{\mathrm D}

\newcommand{\HHH}{\mathscr{H}}

\DeclareMathOperator{\Cr}{Cr}
\DeclareMathOperator{\Aut}{Aut}
\DeclareMathOperator{\Out}{Out}

\DeclareMathOperator{\GL}{GL}
\DeclareMathOperator{\SL}{SL}
\DeclareMathOperator{\PGL}{PGL}
\DeclareMathOperator{\PSL}{PSL}
\DeclareMathOperator{\PU}{PU}
\DeclareMathOperator{\PSU}{PSU}

\def \ge {\geqslant}
\def \le {\leqslant}

\title{On $p$-Jordan constant of Cremona group of rank~$2$ in odd characteristic}
\author{Yifei Chen and Constantin Shramov}

\address{\emph{Yifei Chen}
\newline
\textnormal{ State Key Laboratory of Mathematical Sciences, Academy of Mathematics and Systems Science, Chinese Academy of Sciences, No. 55 Zhonguancun East Road, Haidian District, Beijing, 100190, P.R.China.
}
\newline
\textnormal{\texttt{yifeichen@amss.ac.cn}}
}

\address{\emph{Constantin Shramov}
\newline
\textnormal{Steklov Mathematical Institute of RAS,
8 Gubkina street, Moscow, 119991, Russia.
}
\newline
\textnormal{
National Research University Higher School of Economics,  6~Usacheva str., Moscow, 117312, Russia.
}
\newline
\textnormal{\texttt{costya.shramov@gmail.com}}}

\thanks{Yifei Chen was supported by the NSFC Grant (No. 12271384).
The work of Constantin Shramov was performed at the Steklov International Mathematical Center and supported by the Ministry of Science and Higher Education of the Russian Federation (agreement no.~\mbox{075-15-2022-265}) and by
the HSE University Basic Research Program.}

\begin{document}
\maketitle

\begin{abstract}
We bound the indices of normal abelian subgroups in finite
groups contained in the Cremona group of rank~$2$ over a field of odd
characteristic.
\end{abstract}

\tableofcontents

\section{Introduction}

A classical theorem of C.\,Jordan
asserts that finite subgroups
in a general linear group
over a field of characteristic zero contains normal abelian subgroups
of bounded index.

\begin{theorem}[{see e.g. \cite[\S40]{Jordan} or \cite[Theorem~9.9]{Serre-FiniteGroups}}]
\label{theorem:Jordan}
Let $\Bbbk$ be a field of characteristic zero,
and let $n$ be a positive integer.
Then there exists a constant~$J(n)$ such that
every finite subgroup of the group $\GL_n(\Bbbk)$
contains a normal abelian subgroup of index at most~$J(n)$.
\end{theorem}

Motivated by Theorem~\ref{theorem:Jordan}, the following definition
was introduced in~\cite{Popov}.

\begin{definition}[{see \cite[Definition~2.1]{Popov}}]
\label{definition:Jordan}
A group~$\Gamma$ is called \emph{Jordan}
(alternatively, one says that~$\Gamma$ has \emph{Jordan property}),
if there exists a constant~\mbox{$J=J(\Gamma)$},
depending only on~$\Gamma$,
such that any finite subgroup of $\Gamma$
contains a normal abelian subgroup of index at most~$J$.
\end{definition}

Thus, Theorem~\ref{theorem:Jordan} claims that the group
$\GL_n(\Bbbk)$ (and thus also every linear algebraic group)
over a field $\Bbbk$ of characteristic zero is Jordan.
According to~\mbox{\cite[Theorem~1.6]{MengZhang}}, this property also holds
for automorphism groups of projective varieties over a field of characteristic zero.
Another important example of a Jordan group is the Cremona group
$\Cr_2(\Bbbk)$ of rank~$2$, that is,
the birational automorphism group of the projective plane,
over a field $\Bbbk$ of characteristic zero,
see~\mbox{\cite[Theorem~5.3]{Serre-2009}}. Furthermore,
birational automorphism groups of higher-dimensional projective spaces
over a field of characteristic zero
are also known to be Jordan, see~\mbox{\cite[Theorem~1.8]{ProkhorovShramov-Cr}} and~\cite{Birkar}.

The assertion of Theorem~\ref{theorem:Jordan} does not hold
if $\Bbbk$ is an algebraically closed field of positive characteristic,
see e.g.~\mbox{\cite[\S1]{ChenShramov}}.
However, the following weaker property
holds quite often for groups of geometric origin over such fields.
In the sequel, for a finite group $G$ and a prime number $p$, we will denote by $G_{(p)}$
the $p$-Sylow subgroup of~$G$.

\begin{definition}[{\cite[Definition~1.6]{Hu}}]
\label{D:g-Jordan-2}
Let $p$ be a prime number. A group $\Gamma$ is called
$p$-\emph{Jordan}, if there exist constants
$J(\Gamma)$ and~\mbox{$e(\Gamma)$}, depending only on $\Gamma$, such that every finite subgroup
$G$ of $\Gamma$ contains a normal abelian subgroup of order coprime
to $p$ and index at most~\mbox{$J(\Gamma)\cdot |G_{(p)}|^{e(\Gamma)}$}.
\end{definition}

\begin{theorem}[{see \cite[Theorem~0.4]{LarsenPink}, cf. \cite{BrauerFeit}, \cite{Weisfeiler}, \cite{BajpaiDona}}]
\label{theorem:BrauerFeit}
Let $n$ be a positive integer. Then there
exists a constant $J(n)$ such that for every prime $p$ and every
field $\Bbbk$ of characteristic $p$, every finite subgroup
$G$ of~\mbox{$\GL_n(\Bbbk)$}
contains a normal abelian subgroup of order coprime
to $p$ and index at most~\mbox{$J(n)\cdot |G_{(p)}|^{3}$}.
In particular, for every field $\Bbbk$ of characteristic $p>0$
the group~\mbox{$\GL_n(\Bbbk)$} is $p$-Jordan.
\end{theorem}

\begin{theorem}[{\cite[Theorem~1.10]{Hu}}]
\label{theorem:Hu}
Let $\Bbbk$ be a field of characteristic $p>0$, and
let $X$ be a projective variety over $\Bbbk$.
Then the automorphism group $\Aut(X)$ is $p$-Jordan.
\end{theorem}

%

The $p$-Jordan property is  known for Cremona groups of rank $2$ over a field of positive characteristic~$p$.

\begin{theorem}[{\cite[Theorem~1.6]{ChenShramov}}]
\label{theorem:ChenShramov}
There exists a constant $J$ such that for every prime $p$ and every field $\Bbbk$ of characteristic $p$, every finite subgroup~$G$
of~\mbox{$\Cr_2(\Bbbk)$}
contains a normal abelian subgroup of order coprime
to $p$ and index at most~\mbox{$J\cdot |G_{(p)}|^{3}$}.
In particular, for every field $\Bbbk$ of characteristic $p>0$
the group~\mbox{$\Cr_2(\Bbbk)$} is $p$-Jordan.
\end{theorem}

The minimal possible value of the constant $J(n)$ appearing in Theorem~\ref{theorem:Jordan}
was found in~\cite{Collins} for every positive integer~$n$.
Also, the following result is available.

\begin{theorem}[{\cite[Theorem~1.9]{Yasinsky}}]
\label{theorem:Yasinsky}
For every field $\Bbbk$ of characteristic zero, every finite subgroup
of~\mbox{$\Cr_2(\Bbbk)$}
contains a normal abelian subgroup of index at most~$7200$.
\end{theorem}

Recently, an analog of Theorem~\ref{theorem:Yasinsky}
with an explicit bounding constant for an arbitrary field of characteristic zero
(so that the constant depends on arithmetic properties of the field) was
obtained in~\cite{Zaitsev}. Note also
that while Jordan property does not hold for the group $\Cr_2(\Bbbk)$
for an algebraically closed field~$\Bbbk$ of positive characteristic, it does
hold when~$\Bbbk$ is a finite field by~\cite{ProkhorovShramov}.
Furthermore, the relevant constants are known in this case.
We will denote by~$\mathbf{F}_q$ the finite field of~$q$ elements.

\begin{theorem}[{see \cite[Theorem~1.2]{ProkhorovShramov} and \cite[Theorem~1]{Vikulova}}]
Every finite subgroup of~\mbox{$\Cr_2(\mathbf{F}_q)$}
contains a normal abelian subgroup of index at most
$$
J\big(\Cr_2(\mathbf{F}_q)\big)=
\begin{cases}
q^3(q^2-1)(q^3-1)=|\PGL_3(\mathbf{F}_q)|, &\text{if\ } q\ge 3,\\
720>|\PGL_3(\mathbf{F}_2)|, &\text{if\ } q=2.
\end{cases}
$$
\end{theorem}

The purpose of this paper is to prove the following result which makes
Theorem~\ref{theorem:ChenShramov} more precise in the case of fields of characteristic different from~$2$.

\begin{theorem}
\label{theorem:main}
For every field $\Bbbk$ of odd characteristic $p$,
every finite subgroup~$G$
of~\mbox{$\Cr_2(\Bbbk)$}
contains a normal abelian subgroup of order coprime
to $p$ and index at most~\mbox{$J_p(\Cr_2)\cdot |G_{(p)}|^{3}$},
where
$$
J_p(\Cr_2)=
\begin{cases}
7200, &\text{if\ } p\ge 7,\\
168, &\text{if\ } p=5,\\
10, &\text{if\ } p=3.
\end{cases}
$$
Moreover, this bound is sharp if the field $\Bbbk$ is algebraically closed.
In other words, for any constant $J'<J_p(\Cr_2)$
there exists a finite subgroup in $\Cr_2(\Bbbk)$ which does not
contain normal abelian subgroups of
order coprime to $p$ and index at most~\mbox{$J'\cdot |G_{(p)}|^{3}$};
and for any constants $J''$ and $e<3$
there exists a finite subgroup in $\Cr_2(\Bbbk)$ which does not
contain normal abelian subgroups of
order coprime to $p$ and index at most~\mbox{$J''\cdot |G_{(p)}|^{e}$}.
\end{theorem}

Let us explain the structure of our proof of Theorem~\ref{theorem:main}.
As usual, the question about finite subgroups of Cremona group
can be reduced to a similar question about finite groups acting
on conic bundles and del Pezzo surfaces. Groups acting on conic bundles
can be understood once we know which finite groups act on~$\PP^1$; the latter
groups were classified long ago, see Theorem~\ref{theorem:ADE}.
Among del Pezzo surfaces of high degree there is the projective plane,
whose automorphism group $\PGL_3(\Bbbk)$ is well studied. We rely on the classification
of finite subgroups of $\PGL_3(\Bbbk)$ in odd characteristic, see Theorem~\ref{theorem:Mitchell}.
However, in characteristic $2$ a similar classification is not available: only the list of \emph{maximal}
finite subgroups is known. We point out here that, unlike the bounds in the case of the usual
Jordan property, the bounds in the case of $p$-Jordan groups do not behave well
under passing to subgroups, see Examples~\ref{example:constant-subgroup-1} and~\ref{example:constant-subgroup-2}. Thus, if one wants to study such bounds
for finite subgroups of $\PGL_3(\Bbbk)$ starting from the list of maximal finite subgroups,
further group-theoretic work is needed, which we did not manage to do yet.
Anyway, returning to del Pezzo surfaces, one can deal with them using the recent
works~\cite{DD}, \cite{DM-odd}, and~\cite{DM-2}
of I.\,Dolgachev, A.\,Duncan, and G.\,Martin, which
provide a complete classification of automorphism groups for degrees at most~$5$.
However, we take a slightly longer way to avoid
using the main results of these works, which are rather lengthy and
rely on complicated computations.
Instead, we analyze the groups acting on
surfaces of low degree ($1$ and~$2$) using the properties of their anticanonical
linear systems. For del Pezzo surfaces
of degree $3$ and $4$ we do not use geometric properties of the surfaces at all,
bounding the indices of normal abelian subgroups in the subgroups of the corresponding
Weyl groups which satisfy only the simplest restrictions one can obtain from geometry.

It would be interesting to find a bound similar to those in Theorem~\ref{theorem:main}
in the case of characteristic~$2$; see Remark~\ref{remark:future} below
for a more detailed discussion of this case.
Also, it would be interesting to find the values of the constants~$J(n)$ appearing in
Theorem~\ref{theorem:BrauerFeit}. Some constants related to them were computed
in~\cite{Collins-modular}.

\smallskip

The plan of the paper is as follows.
In Section~\ref{section:examples}
we find the bounds for the indices of normal abelian
subgroups in several finite groups that will be used later.
In Section~\ref{section:preliminaries}
we collect several auxiliary group-theoretic results.
In Section~\ref{section:PSL-PGL}
we collect information about projective linear groups.
In Section~\ref{section:semi-direct}
we discuss certain semi-direct products
and their subgroups. In Section~\ref{section:extensions}
we bound the indices of normal abelian subgroups in certain extensions of cyclic groups.
In Section~\ref{section:P1}
we bound the indices of normal abelian subgroups of finite
groups acting on~$\PP^1$ and~$\PP^1\times\PP^1$.
In Section~\ref{section:P2}
we give preliminary bounds for the indices of normal abelian subgroups of finite
groups acting on~$\PP^2$
over fields of odd characteristic.
In Section~\ref{section:CB}
we bound the indices of normal abelian subgroups of finite
groups acting on conic bundles, and use these bounds to make final conclusions
about finite groups acting on~$\PP^2$ in odd characteristic.
In Section~\ref{section:dP}
we bound the indices of normal abelian subgroups of finite
groups acting on del Pezzo surfaces.
In Section~\ref{section:proof}
we bring everything together and
complete the proof of Theorem~\ref{theorem:main}; in particular, we produce examples
showing that our bounds are attained.

Throughout the paper we use the following terminology and notation.
A del Pezzo surface is a smooth geometrically irreducible projective
surface $S$ such that its anticanonical divisor $-K_S$ is ample; the degree of $S$
is the self-intersection~$K_S^2$ of its (anti)canonical divisor.
We denote by $\mumu_n$ the cyclic group of order $n$.
The symmetric and alternating groups of degree $n$ are denoted by~$\mathfrak{S}_n$
and~$\mathfrak{A}_n$, respectively.
By $\DD_{2n}$, $n\ge 2$, we denote the dihedral group of order~$2n$,
so that in particular $\DD_4\cong\mumu_2^2$.
By~$\HHH_3$ we denote the Heisenberg group of order~$27$,
i.e. the unique non-abelian group of order~$27$ and exponent~$3$.
By~$\PU_n(\mathbf{F}_q)$ and $\PSU_n(\mathbf{F}_q)$ we denote the images in $\PGL_n(\mathbf{F}_{q^2})$ of the subgroup of all unitary matrices and the subgroup of unitary matrices with determinant~$1$
in~\mbox{$\GL_n(\mathbf{F}_{q^2})$}.
Given a group~$G$, by~$\Out(G)$ we denote the group of outer automorphisms
of $G$, i.e. the quotient of~$\Aut(G)$ by the subgroup formed by inner automorphisms.
If $H$ is a subgroup of $G$, we denote by $\Aut(G;H)$ the group of automorphisms
of $G$ which preserve $H$, i.e. map $H$ isomorphically on itself.
By $C_G(H)$ we denote the centralizer of $H$ in $G$.

We thank Yucong Du and Andrey Vasilyev for numerous useful discussions.
We also thank the referees for valuable comments.

\section{Examples}
\label{section:examples}

In this section we find the bounds for the indices of normal abelian
subgroups in several finite groups that will be used later.

\begin{example}
\label{example:PGL}
Let $p$ be a prime number, let $k$ be a positive integer, and let~\mbox{$G=\PGL_2(\mathbf{F}_{p^k})$}.
Then the index of the trivial subgroup in $G$ equals
$$
|G|=p^k(p^{2k}-1)<p^{3k}=|G_{(p)}|^3.
$$
On the other hand, one cannot choose a constant $J$ and a constant $e<3$
such that $J$ and $e$ do not depend on $k$, and there exists a normal abelian
subgroup in $G$ whose index is at most $J\cdot |G_{(p)}|^e$. Indeed, if $k\ge 2$, then
the group $G$ does not contain non-trivial normal abelian subgroups at all, see
for instance~\mbox{\cite[Lemma 2.4]{ProkhorovShramov}}.
Furthermore, the index of the trivial subgroup is larger than $J\cdot |G_{(p)}|^e$
for any fixed $J$ and $e<3$ provided that $k$ is large enough.
\end{example}

\begin{example}
\label{example:PSL}
Let $p$ be a prime number, let $k$ be a positive integer, and let~\mbox{$G=\PSL_2(\mathbf{F}_{p^k})$}.
If $p\neq 2$, then
the trivial subgroup of $G$ has index
$$
|G|=\frac{|\SL_2(\mathbf{F}_{p^k})|}{2}=\frac{p^k(p^{2k}-1)}{2}
<p^{3k}=|G_{(p)}|^3.
$$
If $p=2$, then
the trivial subgroup of $G$ has index
$$
|G|=|\SL_2(\mathbf{F}_{2^k})|=2^k(2^{2k}-1)<2^{3k}=|G_{(2)}|^3.
$$
We will see later in Theorem~\ref{theorem:PSL-PGL-basic}(i)
that $G$ does not contain non-trivial normal subgroups
unless $k=1$ and $p\in\{2,3\}$.
\end{example}

\begin{example}
\label{example:A5}
Let $p$ be a prime number. The group $G=\mathfrak{A}_5$
is simple, and thus does not contain non-trivial normal abelian subgroups.
The index of the trivial subgroup of $G$ can be written
as~\mbox{$|G|\le J\cdot |G_{(p)}|^3$}, where
\begin{equation*}
J=\begin{cases}
60, &\text{if\ } p\ge 7,\\
\frac{60}{|G_{(5)}|^3}=\frac{12}{25}, &\text{if\ } p=5,\\
\frac{60}{|G_{(3)}|^3}=\frac{20}{9}, &\text{if\ } p=3,\\
\frac{60}{|G_{(2)}|^3}=\frac{15}{16}, &\text{if\ } p=2.
\end{cases}
\end{equation*}
\end{example}

Note that if $G$ is a finite group which contains a normal abelian subgroup
of order coprime to $p$ and index at most $J\cdot |G_{(p)}|^e$, then the same
does not necessarily hold for the subgroups of $G$.
Thus, to find the constants $J$ and $e$ as in Definition~\ref{D:g-Jordan-2}
for a $p$-Jordan group $\Gamma$, one has to compute the relevant constants
for \emph{all} finite subgroups of $\Gamma$, unlike the case of the constant involved in
Definition~\ref{definition:Jordan}, where it is enough to do the computations only for
maximal finite subgroups. Similarly, the bound that is valid for $G$ may not hold for its
quotient group.

\begin{example}\label{example:constant-subgroup-1}
The group $G=\mathfrak{A}_5$ contains a
(trivial) normal abelian subgroup of order coprime to $5$ and index at most
$\frac{12}{25}\cdot |G_{(5)}|^3$. On the other hand, its subgroup
$H=\mumu_3$ cannot contain subgroups of index at most
$J\cdot |H_{(5)}|^3=J$ for~\mbox{$J<1$}.
\end{example}

\begin{example}\label{example:constant-subgroup-2}
Let $H=\mumu_7\rtimes \mumu_3$ be the (unique) non-trivial
semi-direct product, and set $G=\mumu_2\times H$.
Then $G$ contains a normal abelian subgroup $\mumu_7$ of order coprime to $2$
and index
$$
6=\frac{3}{4}\cdot |G_{(2)}|^3.
$$
On the other hand, the group
$H$, which can be viewed both as a subgroup and a quotient group
of $G$, cannot contain subgroups of index at most
$J\cdot |H_{(2)}|^3=J$ for $J<1$.
\end{example}

Recall that a subgroup of a group $\Gamma$ is said to be \emph{characteristic},
if it is preserved by all automorphisms of~$\Gamma$.

\begin{example}
\label{example:S4}
Let $p$ be a prime number.
For the symmetric group $G=\mathfrak{S}_4$,
all the proper nontrivial normal subgroups are $\mathfrak{A}_4$ and the abelian subgroup
$$
V_4=\{\mathrm{id},(12)(34),(13)(24),(14)(23)\}.
$$
In particular, there exists a characteristic subgroup isomorphic to $\mumu_2^2$ in~$G$.
Thus, the group $G$ contains a normal abelian
subgroup of order coprime to $p$ and index at most $J\cdot |G_{(p)}|^3$, where
\begin{equation*}
J=\begin{cases}
6, &\text{if\ } p\ge 5,\\
\frac{2}{9}, &\text{if\ } p=3,\\
\frac{3}{256}, &\text{if\ } p=2.
\end{cases}
\end{equation*}

Observe that $G$ does not contain non-trivial cyclic characteristic subgroups.
The index of the trivial subgroup (which is, of course, characteristic) can be written as
$|G|\le J'\cdot |G_{(p)}|^3$, where
\begin{equation*}
J'=\begin{cases}
24, &\text{if\ } p\ge 5,\\
\frac{8}{9}, &\text{if\ } p=3,\\
\frac{3}{64}, &\text{if\ } p=2.
\end{cases}
\end{equation*}
\end{example}

\begin{example}
\label{example:A4}
Let $p$ be a prime number. The only characteristic cyclic subgroup
in the group $G=\mathfrak{A}_4$ is the trivial subgroup.
Its index can be written as~\mbox{$|G|\le J'\cdot |G_{(p)}|^3$}, where
\begin{equation*}
J'=\begin{cases}
12, &\text{if\ } p\ge 5,\\
\frac{4}{9}, &\text{if\ } p=3,\\
\frac{3}{16}, &\text{if\ } p=2.
\end{cases}
\end{equation*}
\end{example}

\section{Preliminaries}
\label{section:preliminaries}

In this section we collect several auxiliary group-theoretic results.
First, recall the following simple fact concerning dihedral groups.

\begin{lemma}\label{lemma:dihedral-group-basic}
Let $G\cong\DD_{2n}$ be the dihedral group of order $2n$ with $n\ge 3$.
Then the cyclic subgroup $G'$ of order $n$ is characteristic in $G$.
Furthermore, the centralizer of $G'$ in $G$ coincides with $G'$.
\end{lemma}

\begin{proof}
The group $G$ is generated by a rotation $\rho$ of order $n$ and a reflection $\sigma$ satisfying the relations
$$
\rho^n=\sigma^2=1, \quad  \rho\sigma=\sigma\rho^{-1}.
$$
We claim that every automorphism $\varphi$ of $G$ maps the cyclic subgroup
$G'$ generated by $\rho$ to itself. Indeed,
observe that all the elements of $G$ can be written either in the form
$\rho^j$ or in the form $\sigma\rho^j$ for some $0\le j\le n-1$.
For every~\mbox{$1\le j\le n-1$}, the element~$\sigma\rho^j$ has order $2$, because
$$
(\sigma\rho^j)^2=\sigma\rho^j\sigma\rho^j=\sigma\sigma\rho^{-j}\rho^j=\sigma^2=1.
$$
On the other hand, the element $\varphi(\rho)$ has order $n\ge 3$,
so that one must have~\mbox{$\varphi(\rho)=\rho^j$} for some $1\le j\le n-1$. This
means that $\varphi(G')=G'$.

Let us show that the centralizer of $G'$ in $G$ coincides with $G'$.
Indeed, if
$$
(\sigma\rho^j)\rho=\rho(\sigma\rho^j)
$$
for some $1\le j\le n-1$, then we have
$$
\sigma\rho^{j+1}=(\sigma\rho^{-1})\rho^j=\sigma\rho^{j-1}.
$$
Cancelling $\sigma\rho^{j-1}$, one obtains $\rho^2=1$, which gives a contradiction, because
the order of $\rho$ equals $n\ge 3$.
\end{proof}

The next result is known as the theorem of Chermak and Delgado.

\begin{theorem}[{see e.g. \cite[Theorem~1.41]{Isaacs}}]
\label{theorem:Chermak-Delgado}
Let $G$ be a finite group. Suppose that~$G$ contains an abelian subgroup
of index~$I$. Then~$G$ contains a characteristic abelian
subgroup of index at most~$I^2$.
\end{theorem}

\begin{corollary}
\label{corollary:Chermak-Delgado}
Let $G$ be a finite group, and let $p$ be a prime number.
Suppose that~$G$ contains an abelian subgroup
of index at most~$J\cdot |G_{(p)}|^e$. Then~$G$ contains a characteristic abelian
subgroup of order coprime to $p$ and index at most~\mbox{$J^2\cdot |G_{(p)}|^{2e+1}$}.
\end{corollary}

\begin{proof}
We know from Theorem~\ref{theorem:Chermak-Delgado} that $G$ contains
a characteristic abelian subgroup $A$ of index at most
$J^2\cdot |G_{(p)}|^{2e}$. Write $A\cong A_{(p)}\times A'$,
where $A'$ is the subgroup of $A$ which consists of all elements whose
order is coprime to~$p$. So, $A'$ is characteristic in $A$, and thus also in~$G$.
The index of $A'$ in $G$ equals
$$
\frac{|G|}{|A'|}=\frac{|G|}{|A|}\cdot |A_{(p)}|\le J^2\cdot |G_{(p)}|^{2e}\cdot |A_{(p)}|
\le J^2\cdot |G_{(p)}|^{2e+1}.
$$
\end{proof}

The next lemma allows to bound the indices of normal abelian subgroups in a group
which is an extension of a group of small order by an abelian group.

\begin{lemma}\label{lemma:large-abelian}
Let $p$ be a prime number, and let $e\ge 1$ be a constant.
Let $G$ be a finite group which fits into an exact sequence
$$
1\to H\to G\to \bar{G}\to 1,
$$
where $H$ is abelian. Set
$$
J=\frac{|\bar{G}|}{|\bar{G}_{(p)}|^e}.
$$
Then $G$ contains a normal abelian subgroup of order coprime to $p$ and index
at most $J\cdot |G_{(p)}|^e$.
\end{lemma}

\begin{proof}
Since $H$ is abelian, it can be written as a product
$H\cong H_{(p)}\times H'$, where~$H'$ is the product of all $q$-Sylow subgroups
of $H$ for $q\neq p$. The group~$H'$ is abelian and has order coprime to $p$.
Furthermore, it consists of all the elements of~$H$ whose orders are coprime to $p$.
This means that it is characteristic in~$H$, and hence normal in $G$.
Its index in $G$ equals
$$
\frac{|G|}{|H'|}=\frac{|H|}{|H'|}\cdot |\bar{G}|=
|H_{(p)}|\cdot |\bar{G}_{(p)}|^e\cdot\frac{|\bar{G}|}{|\bar{G}_{(p)}|^e}\le
|H_{(p)}|^e\cdot |\bar{G}_{(p)}|^e\cdot\frac{|\bar{G}|}{|\bar{G}_{(p)}|^e}=
J\cdot |G_{(p)}|^e.
$$
\end{proof}

The following lemma will be used in Section~\ref{section:P1}
to bound the indices of normal abelian subgroups
acting on~$\PP^1\times\PP^1$.

\begin{lemma}
\label{lemma:product}
Let $p$ be a prime number, and let~$\Gamma_1$ and~$\Gamma_2$ be groups.
Suppose that every finite subgroup
$E$ of $\Gamma_1$ contains a normal abelian subgroup of order coprime to $p$
and index at most $J_1\cdot |E_{(p)}|^e$. Suppose also that every finite
subgroup~$F$ of~$\Gamma_2$ contains a characteristic
abelian subgroup of order coprime to $p$
and index at most $J_2\cdot |F_{(p)}|^e$.
Let $G\subset \Gamma_1\times\Gamma_2$ be a finite subgroup.
Then $G$ contains a normal abelian subgroup of order coprime to $p$
and index at most~\mbox{$J_1\cdot J_2\cdot |G_{(p)}|^e$}.
\end{lemma}

\begin{proof}
Denote by $\pi_i\colon G\to \Gamma_i$ the projections to the factors.
Denote
$$
E=\pi_1(G)\subset\Gamma_1,
$$
and let $H\subset G$ be the kernel of $\pi_1$.
By assumption, the group $E$ contains a normal abelian subgroup~$B$ of order coprime to $p$
and index at most $J_1\cdot |E_{(p)}|^e$.

Set~\mbox{$\tilde{B}=\pi_1^{-1}(B)$}
and
$$
F=\pi_2(\tilde{B})\subset \Gamma_2.
$$
Observe that the kernel of $\pi_2$ has trivial
intersection with $H$; in other words, $H$
embeds as a subgroup into $F$. Thus, there is a diagram of homomorphisms, where
we denote by $\tilde{\pi}_i$ the restriction of $\pi_i$ to $\tilde{B}$.
$$
\xymatrix{
H\ar@{^{(}->}[r]\ar@{_{(}->}[rd] &
\tilde{B}\ar@{->>}[r]^{\tilde{\pi}_1}\ar@{->>}[d]^{\tilde{\pi}_2} & B\\
& F &\\
}
$$
Since the order of $B\cong \tilde{B}/H$ is coprime to $p$, we see that
$$
|H_{(p)}|=|\tilde{B}_{(p)}|=|F_{(p)}|.
$$
By assumption, the group $F$ contains a characteristic abelian subgroup~$C$ of order coprime to $p$
and index at most
$$
J_2\cdot |F_{(p)}|^e=J_2\cdot |H_{(p)}|^e.
$$

Set $\tilde{C}=\pi_2^{-1}(C)\subset G$
and
$$
A=\tilde{\pi}_2^{-1}(C)=\tilde{B}\cap \tilde{C}.
$$
Then $\pi_1(A)\subset B$ and $\pi_2(A)\subset C$,
so that $A\subset B\times C$, which means that $A$ is abelian, and the order of $A$ is coprime to~$p$. Furthermore, the index of $A$ in $G$ is
$$
\frac{|G|}{|A|}=\frac{|G|}{|\tilde{B}|}\cdot\frac{|\tilde{B}|}{|A|}=
\frac{|E|}{|B|}\cdot \frac{|F|}{|C|}\le
J_1\cdot |E_{(p)}|^e\cdot J_2\cdot |H_{(p)}|^e=
J_1\cdot J_2\cdot |G_{(p)}|^e.
$$

It remains to check that $A$ is normal in $G$. Since $B$ is normal in $E=\pi_1(G)$,
we see that $\tilde{B}$ is normal in $G$. Thus $F$ is normal in $\pi_2(G)$.
Since $C$ is characteristic in $F$, it is normal in $\pi_2(G)$, and hence $\tilde{C}$
is normal in $G$. Therefore, $A$ is normal as an intersection of the normal
subgroups $\tilde{B}$ and $\tilde{C}$ of $G$.
\end{proof}

\begin{remark}
An analog of Lemma~\ref{lemma:product} with the assumption only about the indices of \emph{normal}
(not necessarily characteristic)  abelian subgroups in finite subgroups of $\Gamma_1$ and $\Gamma_2$ was proved in~\mbox{\cite[Lemma~2.1(2)]{Hu}}. However, the bound which is proved in~\cite{Hu} is also weaker
than that of Lemma~\ref{lemma:product}. It would be
interesting to figure out whether the assertion of Lemma~\ref{lemma:product}
holds with assumptions only about normal abelian subgroups.
\end{remark}

The next result can be found, for instance,
in Corollary~2 of Theorem~2.1.17 in~\cite{Suzuki}.

\begin{theorem}
\label{theorem:normal-in-a-p-group}
Let $p$ be a prime number, and let $G$ be a group of order~$p^n$.
Then~$G$ contains a normal abelian subgroup of order $p^m$ for some
positive integer~$m$ such that $m(m+1)\ge 2n$.
\end{theorem}

We conclude this section by an auxiliary lemma that will be used
in Section~\ref{section:dP} to analyze automorphism groups of cubic surfaces.
Recall that~$\HHH_3$ denotes the Heisenberg group of order~$27$, i.e. the unique non-abelian group of order~$27$ and exponent~$3$..

\begin{lemma}\label{lemma:auxiliary-subgroups}
Let $p$ be a prime number, and let $G$ be a finite group.
Set
\begin{equation}\label{eq:auxiliary-subgroups}
J=\begin{cases}
720, &\text{if\ } p\ge 7,\\
144, &\text{if\ } p=5,\\
10, &\text{if\ } p=3,\\
3, &\text{if\ } p=2.
\end{cases}
\end{equation}
The following assertions hold.
\begin{itemize}
\item[(i)] Suppose that $G\subset\mathfrak{S}_5$. Then $G$ contains a normal abelian subgroup
of order coprime to $p$ and index at most $J\cdot |G_{(p)}|^3$.
Moreover, one has~\mbox{$|G|\le J\cdot |G_{(p)}|^3$}, unless either $p=3$ and $G\cong\mumu_5\rtimes\mumu_4$, or $p=2$ and~\mbox{$G\cong\mumu_5$}.

\item[(ii)] Suppose that $G\subset \mumu_2^4\rtimes\mathfrak{S}_5$, where
the action of $\mathfrak{S}_5$ on $\mumu_2^4$ comes from the permutation representation.
Then $G$ contains a normal abelian subgroup
of order coprime to $p$ and index at most $J\cdot |G_{(p)}|^3$,
unless $p=3$ and~\mbox{$G\cong\mumu_2^4\rtimes (\mumu_5\rtimes\mumu_4)$}.

\item[(iii)] Suppose that $G\subset \mumu_2^4\rtimes\mathfrak{A}_5$, where
the action of $\mathfrak{A}_5$ on $\mumu_2^4$ comes from the permutation representation.
Then $G$ contains a normal abelian subgroup
of order coprime to $p$ and index at most $J\cdot |G_{(p)}|^3$.

\item[(iv)] Suppose that $G\subset\mathfrak{S}_6$.
Then $G$ contains a normal abelian subgroup
of order coprime to $p$ and index at most $J\cdot |G_{(p)}|^3$.

\item[(v)] Suppose that $G\subset\Gamma=\HHH_3\rtimes\SL_2(\mathbf{F}_3)$.
Then $G$ contains a normal abelian subgroup
of order coprime to $p$ and index at most $J\cdot |G_{(p)}|^3$,
unless $p=5$ and either $G=\Gamma$, or $|G|=162$.

\item[(vi)] Suppose that $G\subset\mumu_3^3\rtimes \mathfrak{S}_4$.
Then $G$ contains a normal abelian subgroup
of order coprime to $p$ and index at most $J\cdot |G_{(p)}|^3$.

\item[(vii)] Suppose that $G\subset\Gamma$, where $\Gamma$ is a group of order $576$.
Then $G$ contains a normal abelian subgroup
of order coprime to $p$ and index at most~\mbox{$J\cdot |G_{(p)}|^3$},
unless $p=5$ and $|G|\in\{192,288,576\}$.
\end{itemize}
\end{lemma}

\begin{proof}
Suppose that $G\subset \mathfrak{S}_5$. If $p\ge 5$, then $|G|<J$, so that
there is nothing to prove. Thus, we have to consider only the cases $p=2$ and $p=3$.
One has
$$
|G|\in\{1,2,3,4,5,6,8,10,12,20,24,60,120\}.
$$
Checking the orders of the $2$- and $3$-Sylow subgroups case by case, we see that the
inequality $|G|\le J\cdot |G_{(p)}|^3$ fails only if $p=3$ and $|G|=20$, and if
$p=2$ and~\mbox{$|G|=5$}. In the former case one has $G\cong\mumu_5\rtimes\mumu_4$,
and in the latter case~\mbox{$G\cong\mumu_5$}. This proves the second part of assertion~(i).
Now observe that if the bound~\mbox{$|G|\le J\cdot |G_{(p)}|^3$} holds, then the first part of assertion~(i)
holds as well. On the other hand, if $p=3$ and $G\cong\mumu_5\rtimes\mumu_4$, then
$G$ contains a normal abelian subgroup $\mumu_5$ of index $4<J$.
This completes the proof of assertion~(i).

Suppose that $G$ satisfies the assumptions of assertion (ii). Then $G$ fits into
an exact sequence
$$
1\to H\to G\to\bar{G}\to 1,
$$
where $H\subset \mumu_2^4$ and $\bar{G}\subset\mathfrak{S}_5$.
If $p\ge 3$, then it follows from assertion~(i) that
$$
\frac{|G|}{|H|}=|\bar{G}|\le J\cdot |\bar{G}_{(p)}|^3= J\cdot |G_{(p)}|^3,
$$
unless $p=3$ and $\bar{G}\cong\mumu_5\rtimes \mumu_4$. Observe that
$\mumu_2^4$ does not contain proper non-trivial $\mumu_5$-invariant subgroups.
This means that if $\bar{G}\cong\mumu_5\rtimes \mumu_4$, then we have either $G\cong\mumu_5\rtimes\mumu_4$, so that $G$ contains a normal subgroup $\mumu_5$
of index~\mbox{$4<J$}, or $G\cong\mumu_2^4\rtimes (\mumu_5\rtimes\mumu_4)$.
Thus, we are left with the case $p=2$.
In this case assertion~(i) implies that
$$
|G|=|H|\cdot |\bar{G}|=|H_{(2)}|\cdot |\bar{G}|\le |H_{(2)}|^3\cdot  J\cdot |\bar{G}_{(2)}|^3=
J\cdot |G_{(2)}|^3,
$$
unless $\bar{G}\cong \mumu_5$.
Since $\mumu_2^4$ does not contain proper non-trivial $\mumu_5$-invariant subgroups,
we see that in the latter case we have either $G\cong\mumu_5$, so that
$G$ is abelian and has order coprime to~$2$, or
$G\cong\mumu_2^4\rtimes \mumu_5$, so that
$$
|G|=40 < 3\cdot 2^{12}=J\cdot |G_{(2)}|^3.
$$
This proves assertion~(ii). Assertion~(iii) is a direct consequence of assertion~(ii).

Suppose that $G\subset\mathfrak{S}_6$.
If $p\ge 7$, then $|G|\le J$. If $p=5$ and $G$ is either $\mathfrak{S}_6$ or $\mathfrak{A}_6$,
then $|G|<J\cdot |G_{(5)}|^3$. If $G$ is neither $\mathfrak{S}_6$ nor $\mathfrak{A}_6$,
then $|G|\le 120$; in this case we have $|G|< J$ for $p=5$.
Thus, we have to consider only the cases $p=2$ and $p=3$.
One has
$$
|G|\in\{1,2,3,4,5,6,8,9,10,12,16,18,20,24,36,48,60,72,120,360,720\}.
$$
Checking the orders of the $2$- and $3$-Sylow subgroups case by case, we see that the
inequality $|G|\le J\cdot |G_{(p)}|^3$ fails only if $p=3$ and $|G|\in\{16,20\}$,
and if~\mbox{$p=2$} and $|G|\in\{5,9\}$. In the former case one has either $G\cong \mumu_2\times\DD_8$,
so that~$G$ contains a normal abelian subgroup $\mumu_2\times\mumu_4$ of index $2<J$,
or $G\cong \mumu_5\rtimes\mumu_4$, so that $G$ contains a normal abelian subgroup $\mumu_5$ of index $4<J$. In the latter case one has either $G\cong \mumu_5$, or $G\cong \mumu_3^2$;
in both cases $G$ is abelian and has order coprime to $2$.
This proves assertion~(iv).

Suppose that $G\subset \Gamma\cong\HHH_3\rtimes\SL_2(\mathbf{F}_3)$.
If $p\ge 7$, then $|G|<J$, so that there is nothing to prove.
The group $G$ fits into
an exact sequence
$$
1\to H\to G\to\bar{G}\to 1,
$$
where $H\subset \HHH_3$ and $\bar{G}\subset\SL_2(\mathbf{F}_3)$.
Suppose that $p=5$ and $G\neq \Gamma$; then~\mbox{$|G|\le 324$}.
If $H$ contains the center $Z\cong\mumu_3$ of $\HHH_3$, then $Z$ is a normal subgroup
of $G$, and its index in $G$ is
$$
\frac{|G|}{|Z|}\le \frac{324}{3}=108<J.
$$
If $H$ does not contain $Z$, then $H$ is abelian, and its index in $G$
is at most~\mbox{$|\bar{G}|=24<J$}. Thus, we have to consider only the cases $p=2$ and $p=3$.
One has
$$
|\bar{G}|\in\{1,2,3,4,6,8,24\}.
$$
If $p=3$, we conclude that $|\bar{G}|<J\cdot |\bar{G}_{(3)}|^3$, which means that
$$
|G|=|H|\cdot |\bar{G}|=|H_{(3)}|\cdot |\bar{G}|<|H_{(3)}|^3\cdot J\cdot |\bar{G}_{(3)}|^3=
J\cdot |G_{(3)}|^3.
$$

Thus, to prove assertion~(v) one can suppose that $p=2$. In this case we see that
$|\bar{G}|\le J\cdot |\bar{G}_{(2)}|^3$. Hence, if $H$ is abelian, then
we are done. Suppose that~$H$ is not abelian.
Then~\mbox{$H=\HHH_3$}, and $Z$ is a normal abelian subgroup of~$G$ of index
$$
9\cdot |\bar{G}|<J\cdot |\bar{G}_{(2)}|^3,
$$
unless $|\bar{G}|\in\{1,3,6\}$. If $|\bar{G}|=1$, then $G\cong \HHH_3$ contains a normal
abelian subgroup of index $3=J$. If $|\bar{G}|=3$, then $|G|=81$, and hence
$G$ contains a normal abelian subgroup of index $3$ by Theorem~\ref{theorem:normal-in-a-p-group}.
This completes the proof of assertion~(v).

Suppose that $G\subset\mumu_3^3\rtimes \mathfrak{S}_4$.
Then $G$ fits into
an exact sequence
$$
1\to H\to G\to\bar{G}\to 1,
$$
where $H\subset \mumu_3^3$ and $\bar{G}\subset\mathfrak{S}_4$.
If $p\neq 3$, then the index of the normal abelian subgroup $H$ in $G$
is at most
$$
|\bar{G}|\le J\cdot |\bar{G}_{(p)}|^3=J\cdot |G_{(p)}|^3
$$
by assertion~(i).
Thus, we have to consider only the case $p=3$.
One has
$$
|\bar{G}|\in\{1,2,3,4,6,8,12,24\}.
$$
In each of these cases we have $|\bar{G}|< J\cdot |\bar{G}_{(3)}|^3$,
so that
$$
|G|=|H|\cdot |\bar{G}|=|H_{(3)}|\cdot |\bar{G}|< |H_{(3)}|^3\cdot J\cdot |\bar{G}_{(3)}|^3=
J\cdot |G_{(3)}|^3.
$$
This proves assertion~(vi).

Suppose that $G\subset\Gamma$, where $\Gamma$ is a group of order $576$.
If $p\ge 7$, then $|G|<J$, so that there is nothing to prove.
If $p=5$ and the index of $G$ in $\Gamma$ is at least $4$, then $|G|\le 144=J$;
if the index of $G$ in $\Gamma$ is at most $3$, then
$$
|G|\in\{192,288,576\}.
$$
Thus, we have to consider only the cases $p=2$ and $p=3$.
Write $|G|=2^a3^b$, where $a\le 6$ and $b\le 2$. If
$p=3$, then
$$
|G|=10\cdot \frac{2^a}{10\cdot 3^{2b}}\cdot 3^{3b}=
J\cdot \frac{2^{a-1}}{5\cdot 3^{2b}}\cdot |G_{(3)}|^3\le J\cdot |G_{(3)}|^3,
$$
unless $2^{a-1}>5\cdot 3^{2b}$. The latter is possible only if $a\ge 4$ and $b=0$, which
gives~\mbox{$|G|\in\{16,32,64\}$}. In each of these cases $G$ contains a normal abelian subgroup of
order at least $8$ by Theorem~\ref{theorem:normal-in-a-p-group}, so that this
subgroup has index at most~\mbox{$8<J$} in $G$.
If $p=2$, then
$$
|G|=3\cdot \frac{3^{b-1}}{2^{2a}}\cdot 2^{3a}=
J\cdot \frac{3^{b-1}}{2^{2a}}\cdot |G_{(2)}|^3\le J\cdot |G_{(2)}|^3,
$$
unless $3^{b-1}>2^{2a}$. The latter is possible only if $b=2$ and $a\le 1$, which
gives~\mbox{$|G|\in\{9,18\}$}. In the former case $G$ is abelian, while in the latter
case the $3$-Sylow subgroup of $G$ is abelian, has index $2<J$ in $G$, and is normal
in~$G$. This proves assertion~(vii).
\end{proof}

\section{Projective linear groups}
\label{section:PSL-PGL}

In this section we recall basic information about projective
linear groups.

\begin{theorem}
\label{theorem:PSL-PGL-basic}
Let $p$ be a prime number. The following assertions hold.
\begin{itemize}
\item[(i)]
The group $\PSL_2(\mathbf{F}_{p^k})$ is simple,
unless $k=1$ and $p\in\{2,3\}$.

\item[(ii)] There are isomorphisms
\begin{multline*}
\PSL_2(\mathbf{F}_2)\cong\mathfrak{S}_3, \quad
\PSL_2(\mathbf{F}_3)\cong\mathfrak{A}_4, \\
\PSL_2(\mathbf{F}_4)\cong \PSL_2(\mathbf{F}_5)\cong\mathfrak{A}_5, \quad
\mathfrak{A}_6\cong\PSL_2(\mathbf{F}_9).
\end{multline*}

\item[(iii)]
There are isomorphisms
$$
\PGL_2(\mathbf{F}_{2})\cong \mathfrak{S}_3, \quad
\PGL_2(\mathbf{F}_{3})\cong \mathfrak{S}_4.
$$

\item[(iv)]
The center of the group $\PSL_2(\mathbf{F}_{p^k})$, the center of the group
$\PGL_2(\mathbf{F}_{p^k})$, and the centralizer of the subgroup
$\PSL_2(\mathbf{F}_{p^k})$ in $\PGL_2(\mathbf{F}_{p^k})$
are trivial.

\item[(v)] The subgroup $\PSL_2(\mathbf{F}_{p^k})$ is characteristic
in $\PGL_2(\mathbf{F}_{p^k})$.
\end{itemize}
\end{theorem}

\begin{proof}
For assertion~(i), see~\mbox{\cite[\S3.3.2]{Wilson}}.
For assertion~(ii), see~\mbox{\cite[\S3.3.5]{Wilson}}
or~\mbox{\cite[\S10.1]{Serre-FiniteGroups}}.
For assertion~(iii), see~\mbox{\cite[\S10.1]{Serre-FiniteGroups}}.

Let us show that the centralizer of $\PSL_2(\mathbf{F}_{p^k})$ in $\PGL_2(\mathbf{F}_{p^k})$
is trivial. Suppose that $\bar{A}\in\PGL_2(\mathbf{F}_{p^k})$ commutes with every element of $\PSL_2(\mathbf{F}_{p^k})$. This means that for every $B\in\SL_2(\mathbf{F}_{p^k})$, there exists $\lambda\in\mathbf{F}_{p^k}^*$ depending on $B$ such that
$$
AB=\lambda BA,
$$
where $A\in \GL_2(\mathbf{F}_{p^k})$ is a preimage of
$\bar{A}$ under the natural projection.
We claim that $A$ must be a scalar matrix. Indeed, suppose that $A$ is not scalar. Then there exists a vector $v\in \mathbf{F}_{p^k}^2$ such that $v$ and $Av$ are linearly independent.
Let~\mbox{$B\in \SL_2(\mathbf{F}_{p^k})$} be the matrix that fixes $v$ and maps $Av$ to $v+Av$;
in other words, in the basis $\{v,Av\}$ the matrix $B$ takes the form
$$
B'=\left(
\begin{array}{cc}
1 & 1\\
0 & 1
\end{array}
\right).
$$
Then
$$
(AB-\lambda BA)v =ABv-\lambda BAv=Av-\lambda(v+Av)=(1-\lambda)Av-\lambda v\neq 0
$$
for any $\lambda\in\mathbf{F}_{p^k}^*$.
So
$$
AB\neq \lambda BA
$$
for any $\lambda\in\mathbf{F}_{p^k}^*$, which gives a contradiction.

Therefore, we see that the centralizer of $\PSL_2(\mathbf{F}_{p^k})$ in $\PGL_2(\mathbf{F}_{p^k})$ is trivial. This implies that the center of
$\PSL_2(\mathbf{F}_{p^k})$ and the center of $\PGL_2(\mathbf{F}_{p^k})$
is trivial, and proves assertion~(iv).

Observe that $\PSL_2(\mathbf{F}_{2^k})=\PGL_2(\mathbf{F}_{2^k})$, so the former is trivially a characteristic subgroup of the latter.
For $p\ge 3$, we claim that $\PSL_2(\mathbf{F}_{p^k})$ is the commutator subgroup of  $\PGL_2(\mathbf{F}_{p^k})$. Indeed, for
$p=3$ and $k=1$, we have~\mbox{$\PSL_2(\mathbf{F}_3)\cong\mathfrak{A}_4$} and $\PGL_2(\mathbf{F_3})\cong\mathfrak{S}_4$ by assertions~(ii) and~(iii);
in this case the claim is straightforward to check.
If $p\ge 5$, or $p=3$ and $k\ge 2$, then~\mbox{$\PSL_2(\mathbf{F}_{p^k})$} is a non-abelian simple group by assertion~(i).
Since the commutator subgroup of $\PSL_2(\mathbf{F}_{p^k})$ is normal in $\PSL_2(\mathbf{F}_{p^k})$, this implies that
$$
[\PSL_2(\mathbf{F}_{p^k}),\PSL_2(\mathbf{F}_{p^k})]=\PSL_2(\mathbf{F}_{p^k}).
$$
On the other hand, the commutator of any two matrices
from $\GL_2(\mathbf{F}_{p^k})$ has determinant equal to~$1$, which means that
$$
[\PGL_2(\mathbf{F}_{p^k}),\PGL_2(\mathbf{F}_{p^k})]\subset\PSL_2(\mathbf{F}_{p^k}),
$$
and the claim follows. Since the commutator subgroup is always
characteristic, this proves assertion~(v).
\end{proof}

The next theorem describes the outer automorphism group of the
group~\mbox{$\PSL_2(\mathbf{F}_{p^k})$}.

\begin{theorem}[{see e.g. \cite[Theorem~3.2(ii)]{Wilson}}]
\label{theorem:Out-PSL}
If $p\ge 3$, one has
$$
\Out(\PSL_2(\mathbf{F}_{p^k}))\cong \mumu_2\times \mumu_k.
$$
Furthermore, one has
$$
\Out(\PSL_2(\mathbf{F}_{2^k}))\cong \mumu_k.
$$
\end{theorem}

We will also need the following general auxiliary fact.
Recall that~$C_G(H)$ denotes the centralizer of a subgroup~$H$ in a group~$G$.

\begin{lemma}\label{lemma:restriction-Aut}
Let $G$ be a group and $H$ a characteristic subgroup of $G$.
Suppose that the centralizer $C_G(H)$ is trivial.
Then the restriction homomorphism
$$
\Aut(G)\to\Aut(H), \quad \sigma\mapsto \sigma\vert_H,
$$
is injective.
\end{lemma}

\begin{proof}
Suppose that $\sigma \in\Aut(G)$ is contained in the kernel of the restriction
homomorphism, that is,
the automorphism $\sigma$ restricts to the trivial automorphism of $H$.
Since $H$ is normal in $G$, we have $g^{-1}hg\in H$ for any $g\in G$ and~\mbox{$h\in H$}.
Therefore, one has
$$
g^{-1}hg=\sigma\big(g^{-1}hg\big)=\sigma(g)^{-1}\sigma(h)\sigma(g)=\sigma(g)^{-1}h\sigma(g),
$$
because $\sigma$ is trivial on $H$.
Thus, we obtain
$$
h=(g\sigma(g)^{-1})h(\sigma(g)g^{-1})=
(\sigma(g)g^{-1})^{-1}h(\sigma(g)g^{-1})
$$
for all $h\in H$. This means that
$$
\sigma(g)g^{-1}\in C_G(H)
$$
and hence $g=\sigma(g)$ for all~\mbox{$g\in G$}. In other words, $\sigma$ is the trivial automorphism of~$G$.
\end{proof}

Now we are able to give a detailed description of the automorphism groups of the
groups~\mbox{$\PSL_2(\mathbf{F}_{p^k})$} and~\mbox{$\PGL_2(\mathbf{F}_{p^k})$}.

\begin{proposition}
\label{proposition:Aut-PSL-PGL}
Let $p$ be a prime number. Then
$$
\Aut(\PSL_2(\mathbf{F}_{p^k}))\cong
\Aut(\PGL_2(\mathbf{F}_{p^k}))\cong
\PGL_2(\mathbf{F}_{p^k})\rtimes T,
$$
where $\PGL_2(\mathbf{F}_{p^k})$ acts on $\PGL_2(\mathbf{F}_{p^k})$ and $\PSL_2(\mathbf{F}_{p^k})$ by conjugations, and the
group~\mbox{$T\cong\mumu_k$} can be chosen so that it is generated by the Frobenius
automorphism of the field~$\mathbf{F}_{p^k}$.
\end{proposition}

\begin{proof}
By Theorem~\ref{theorem:PSL-PGL-basic}(iv), the group $\PGL_2(\mathbf{F}_{p^k})$ acts faithfully on~\mbox{$\PSL_2(\mathbf{F}_{p^k})$} by conjugations.
Furthermore, the automorphism group $T\cong\mumu_k$ of the field~$\mathbf{F}_{p^k}$ acts by automorphisms of $\PSL_2(\mathbf{F}_{p^k})$.
Thus, the automorphism group of~\mbox{$\PSL_2(\mathbf{F}_{p^k})$} contains the group
$\PGL_2(\mathbf{F}_{p^k})\rtimes T$.
On the other hand, one has~\mbox{$\PGL_2(\mathbf{F}_{2^k})=\PSL_2(\mathbf{F}_{2^k})$}, and
$$
\PGL_2(\mathbf{F}_{p^k})/\PSL_2(\mathbf{F}_{p^k})\cong\mumu_2
$$
for $p\ge 3$. Now Theorem~\ref{theorem:Out-PSL}
tells us that the automorphism group of $\PSL_2(\mathbf{F}_{p^k})$
cannot be larger than $\PGL_2(\mathbf{F}_{p^k})\rtimes T$.
Thus, we have
$$
\Aut(\PSL_2(\mathbf{F}_{p^k}))\cong
\PGL_2(\mathbf{F}_{p^k})\rtimes T.
$$

To describe the automorphism group of $\PGL_2(\mathbf{F}_{p^k})$, recall from
Theorem~\ref{theorem:PSL-PGL-basic}(iv),(v) that Lemma~\ref{lemma:restriction-Aut}
applies to the group $\PGL_2(\mathbf{F}_{p^k})$ and its subgroup $\PSL_2(\mathbf{F}_{p^k})$.
Therefore, we have
$$
\Aut(\PGL_2(\mathbf{F}_{p^k}))\cong \Aut(\PSL_2(\mathbf{F}_{p^k})).
$$
\end{proof}

As a by-product of Proposition~\ref{proposition:Aut-PSL-PGL}, one obtains a description of
the outer automorphism group of the group~\mbox{$\PGL_2(\mathbf{F}_{p^k})$}.

\begin{corollary}
\label{corollary:Out-PGL}
Let $p$ be a prime number. The following assertions hold.
\begin{itemize}
\item[(i)] One has
$$
\Out(\PGL_2(\mathbf{F}_{p^k}))\cong \mumu_k.
$$

\item[(ii)] One has
$$
\Out(\mathfrak{A}_4)\cong \mumu_2,\quad
\Out(\mathfrak{S}_4)=1,\quad
\Out(\mathfrak{A}_5)\cong \mumu_2.
$$
\end{itemize}
\end{corollary}

\begin{proof}
Assertion~(i) is immediate from Proposition~\ref{proposition:Aut-PSL-PGL}.
Assertion~(ii) follows from assertion~(i) and Theorem~\ref{theorem:Out-PSL}
together with Theorem~\ref{theorem:PSL-PGL-basic}(ii),(iii).
\end{proof}

\section{Semi-direct products}
\label{section:semi-direct}

In this section we discuss certain semi-direct products and their subgroups.

\begin{lemma}[{see e.g. \cite[Lemma~2.11]{ChenShramov} or \cite[Corollary~2]{Darafsheh}}]
\label{lemma:Darafsheh}
Let $p$ be a prime number, and let $m$ be a non-negative integer. Let $R$ be a group containing a
normal subgroup~\mbox{$R'\cong (\mathbb{Z}/p\mathbb{Z})^m$}, $m\ge 1$, and let $g$ be an element of~$R$. Then for some positive integer~\mbox{$t\le p^m-1$}, the element $g^t$ commutes with~$R'$.
\end{lemma}

The next assertion is a simple corollary of Lemma~\ref{lemma:Darafsheh}.

\begin{corollary}
\label{corollary:Darafsheh}
Let $p$ be a prime number, and let $R=\mumu_p^m\rtimes L$,
where~\mbox{$L\cong \mumu_n$} with $n$ coprime to $p$.
Set $C=C_R(R_{(p)})$ and  $L'=L\cap C$.
The following assertions hold.
\begin{enumerate}
\item[(i)] The group $C$ is generated by $R_{(p)}$ and $L'$, one has $C\cong R_{(p)}\times L'$,
and~$C$ is a characteristic subgroup of $R$. The group $L'$ consists of all the elements of $C$ whose orders are coprime to~$p$.

\item[(ii)] One has $C_R(C)=C$.

\item[(iii)]  The subgroup $L'$ is characteristic in $R$,
the index of $L'$ in $L$ does not exceed~$p^m$,
and the index of $L'$ in $R$ does not exceed~$p^{2m}$.
\end{enumerate}
\end{corollary}

\begin{proof}
If $m=0$, there is nothing to prove. Thus, we assume that $m\ge 1$.

The subgroup $R_{(p)}$ is characteristic in~$R$,
because it consists of all the elements
of $R$ whose orders equal either $1$ or~$p$.
Thus its centralizer~$C_R(R_{(p)})$
is also a characteristic subgroup of $R$.
Furthermore, $R_{(p)}$ is abelian, and $L'$ is the centralizer of $R_{(p)}$ in $L$.
Hence $C$ is generated by $R_{(p)}$ and $L'$, and one has~\mbox{$C\cong R_{(p)}\times L'$}.
This proves assertion~(i).

Observe that $C_R(C)\subset C_R(R_{(p)})=C$. On the other hand, $C$ is abelian by assertion~(i),
so that $C\subset C_R(C)$.
This gives assertion~(ii).

By Lemma~\ref{lemma:Darafsheh}, the index of~$L'$ in~$L$
is at most~\mbox{$p^m-1<p^m$}; thus, the index of $L'$ in $R$
does not exceed~$p^{2m}$. Moreover, $L'$ is characteristic in $C$: indeed, by assertion~(i)
it consists of all the elements of $C$ whose orders are coprime to $p$.
Since $C$ is characteristic in $R$, we conclude that $L'$ is also characteristic in~$R$. This proves assertion~(iii).
\end{proof}

Let us make an observation on the subgroups of the product of two copies
of a group described in Corollary~\ref{corollary:Darafsheh}.

\begin{lemma}\label{lemma:Darafsheh-product}
Let $p$ be a prime number, and let $R\cong\mumu_p^m\rtimes L$,
where~\mbox{$L\cong \mumu_n$} with $n$ coprime to $p$.
Let $H$ be a subgroup of $R\times R$.
Then $H$ contains
a characteristic abelian subgroup of order coprime to $p$
and index at most~\mbox{$|H_{(p)}|^{3}$}.
\end{lemma}

\begin{proof}
The $p$-Sylow subgroup $R_{(p)}$ of $R$ is normal. Hence the
$p$-Sylow subgroup~\mbox{$(R\times R)_{(p)}$} of $R\times R$ is normal,
which implies that the $p$-Sylow subgroup
$$
H_{(p)}=H\cap (R\times R)_{(p)}
$$
of $H$ is normal as well. Set $T=H/H_{(p)}$.
Then the order of $T$ is coprime to $p$. Therefore,
one has
$$
H\cong H_{(p)}\rtimes T
$$
by Schur--Zassenhaus theorem, see e.g.~\mbox{\cite[Theorem~3.8]{Isaacs}}.

Let $\pi\colon T\to L\times L$ be the composition
of the embedding $T\hookrightarrow H\hookrightarrow R\times R$
with the projection
$$
R\times R\to R\times R/(R\times R)_{(p)}\cong R/R_{(p)}\times R/R_{(p)}\cong L\times L.
$$
Since the order of $T$ is coprime to $p$, we see that the kernel of $\pi$ is trivial.
In other words, $\pi$ embeds $T$ into the group $L\times L$. Therefore, the group $T$ is abelian.

We see that $H$ contains an abelian subgroup $T$ of index equal to $|H_{(p)}|$.
According to Corollary~\ref{corollary:Chermak-Delgado},
this implies that $H$ contains a characteristic abelian subgroup of order coprime to $p$
and index at most~\mbox{$|H_{(p)}|^{3}$}.
\end{proof}

The following consequence of Corollary~\ref{corollary:Darafsheh} will be used in Section~\ref{section:P1}.

\begin{corollary}
\label{corollary:Darafsheh-product}
Let $p$ be a prime number, and let $R\cong\mumu_p^m\rtimes L$,
where~\mbox{$L\cong \mumu_n$} with $n$ coprime to $p$.
Let $G$ be a finite group.
Suppose that $G$ contains a subgroup~$G'$ of index~$2$ such that
$G'\subset R\times R$,
and an element of $G\setminus G'$ acts on~\mbox{$R\times R$} by a composition
of conjugation by some element of $G'$ with the non-trivial
permutation of the factors so that this action preserves
the subgroup~$G'$. Then $G$ contains a normal abelian subgroup of order coprime to $p$
and index at most~\mbox{$J\cdot |G_{(p)}|^{3}$},
where
\begin{equation*}
J=\begin{cases}
2, &\text{if\ } p\ge 3,\\
1, &\text{if\ } p=2.
\end{cases}
\end{equation*}
\end{corollary}

\begin{proof}
According to Lemma~\ref{lemma:Darafsheh-product}, the group
$G'$ contains a characteristic abelian subgroup~$A$ of order coprime to $p$ and
index at most $|G'_{(p)}|^3$.
Since the index of~$G'$ in~$G$ equals $2$, we see that $G'$ is normal in $G$.
Thus, $A$ is a normal subgroup of~$G$.
Observe that $|G_{(p)}|=|G'_{(p)}|$ if $p\ge 3$, and
$|G_{(p)}|=2\cdot |G'_{(p)}|$ if~\mbox{$p=2$}.
Therefore, the index of $A$ in $G$ is
$$
\frac{|G|}{|A|}=2\cdot \frac{|G'|}{|A|}\le
2\cdot |G'_{(p)}|^3\le J\cdot |G_{(p)}|^{3}.
$$
\end{proof}

\section{Extensions}
\label{section:extensions}

In this section we bound the indices of normal abelian subgroups in extensions of cyclic groups by groups of several certain types. These results will be used in
Section~\ref{section:P1}.

We start with extensions of cyclic groups by dihedral groups.
Recall that for a group~$H$ and its subgroup~$F$ we denote by~$\Aut(H;F)$ the group of
automorphisms of~$H$ which preserve~$F$.

\begin{lemma}\label{lemma:small-dihedral-group-extension}
Let $p\ge 3$ be a prime number, let $\bar{H}$ be a finite cyclic group of order coprime to~$p$,
and let~\mbox{$F=\mumu_2^2$}.
Let $H$ be a group which fits into an exact sequence
$$
1\to F\to H\to \bar{H}\to 1.
$$
Then~$H$ contains an abelian subgroup of order coprime to $p$
and index at most~$3$ preserved by~\mbox{$\Aut(H;F)$}.
\end{lemma}

\begin{proof}
Let $\alpha$ be an element of $H$ whose image in $\bar{H}$ generates~$\bar{H}$.
Since $|\bar{H}|$ is coprime to $p$, we may replace $\alpha$ by its appropriate power and assume that its order is coprime to $p$.
Conjugation by $\alpha$ gives an automorphism of $F$.
Observe that $\Aut(F)\cong\mathfrak{S}_3$.
Therefore, there exists $t\in\{2,3\}$ such that
the element $\alpha^t$ commutes with $F$.

Let $A\subset H$ be the subgroup generated by $F$ and $\alpha^t$. Then $A$ is abelian, and its order is coprime to $p$.
Let $\bar{A}$ be the image of~$A$ in $\bar{H}$. Then the index of $\bar{A}$
in $\bar{H}$ is at most $t$. On the other hand, $A$ is the preimage
of $\bar{A}$ in $H$, so that
$$
\frac{|H|}{|A|}=\frac{|\bar{H}|}{|\bar{A}|}\le t.
$$
Finally, it remains to notice that $A$ is preserved by $\Aut(H;F)$,
because $\bar{A}$ is preserved by the action of $\Aut(H;F)$
on~$\bar{H}$.
\end{proof}

\begin{lemma}\label{lemma:dihedral-group-extension}
Let $p$ be a prime number, let $\bar{H}$ be a finite cyclic group,
and let~\mbox{$F\cong\DD_{2n}$} be the dihedral group of order $2n$ with $n\ge 3$.
Suppose that $n$ and~$|\bar{H}|$ are coprime to~$p$.
Let $F'\subset F$ be the characteristic
cyclic subgroup of order~$n$.
Let $H$ be a group which fits into an exact sequence
$$
1\to F\to H\to \bar{H}\to 1.
$$
Suppose that for every element $\alpha\in H$
the element $\alpha^2$ commutes with~$F'$.
Then~$H$ contains an abelian subgroup of order coprime to $p$
and index at most~$4$
preserved by~\mbox{$\Aut(H;F)$}.
\end{lemma}

\begin{proof}
Let $\alpha$ be an element of $H$ whose image in $\bar{H}$ generates~$\bar{H}$.
Since $|\bar{H}|$ is coprime to $p$, we may replace $\alpha$ by its appropriate power and assume that its order is coprime to $p$.
By assumption, the element $\alpha^2$ commutes with $F'$.

Let $A\subset H$ be the subgroup generated by $F'$ and $\alpha^2$. Then $A$ is abelian, and its order is coprime to $p$.
Let $\bar{A}$ be the image of~$A$ in $\bar{H}$.
Then $\bar{A}$ is a subgroup of index at most $2$ in $\bar{H}$.
Thus the index of~$A$ in~$H$ equals
$$
\frac{|H|}{|A|}=\frac{|\bar{H}|}{|\bar{A}|}\cdot \frac{|F|}{|F\cap A|}\le 2\cdot \frac{|F|}{|F'|}=4.
$$

It remains to show that $A$ is preserved by $\Aut(H;F)$.
Let $\tilde{A}\subset H$ denote the preimage of $\bar{A}$. Then
$\tilde{A}$ is preserved by $\Aut(H;F)$.
Furthermore, the centralizer~$C_H(F')$
is preserved by $\Aut(H;F)$, because $F'$ is a characteristic
subgroup of~$F$. One has
$$
A\subset \tilde{A}\cap C_H(F')
$$
by construction. On the other hand, if $\beta$ is an element
of $\tilde{A}\cap C_H(F')$, then, multiplying it by an appropriate power of $\alpha^2$, we may assume that $\beta\in F$. Thus, $\beta$ is contained in
$$
F\cap C_H(F')=C_F(F').
$$
The latter intersection coincides with $F'$ by Lemma~\ref{lemma:dihedral-group-basic}.
Hence we have~\mbox{$\beta\in A$}, which implies that
$A=\tilde{A}\cap C_H(F')$. Therefore, $A$ is preserved by $\Aut(H;F)$.
\end{proof}

We will need the following auxiliary result.

\begin{lemma}\label{lemma:extension-auxiliary}
Let $p$ be a prime number, let $\bar{H}$ be a finite cyclic group,
and let~$F$ be a finite group with trivial center. Let $H$ be a group which fits into an exact sequence
$$
1\to F\to H\to \bar{H}\to 1.
$$
Let $\alpha\in H$ be an element of order coprime to $p$, let $B$ be a subgroup
of $H$ generated by $F$ and $\alpha$, and let $\gamma$ be an element of $B$
which commutes with~$F$.
Then
\begin{itemize}
\item[(i)] the order of $\gamma$ is coprime to $p$;

\item[(ii)] the cyclic group $A$ generated by $\gamma$ is
preserved by $\Aut(H;F)$.
\end{itemize}
\end{lemma}

\begin{proof}
Let $m$ be the order of~$\alpha$. Then $m$ is coprime to $p$ by assumption.
Denote by $\bar{\alpha}$ and $\bar{\gamma}$ the images of $\alpha$ and $\gamma$ in $\bar{H}$, and let $\bar{m}$ and $\bar{n}$ be the orders of $\bar{\alpha}$
and $\bar{\gamma}$. Thus, $\bar{m}$ divides $m$, and so is coprime to $p$. Observe that $\bar{\gamma}$ is a power of $\bar{\alpha}$, so that $\bar{n}$ divides $\bar{m}$, which
means that $\bar{n}$ is also coprime to~$p$.

The element
$$
\theta=\gamma^{\bar{n}}
$$
is contained in $F$ and commutes with $F$. Since the center of $F$ is trivial, we conclude that $\theta=1$, and thus the order of $\gamma$ divides
$\bar{n}$. On the other hand, $\bar{n}$ divides the order of $\gamma$ by construction,
so that the order of $\gamma$ actually equals $\bar{n}$. This proves assertion~(i).

To show that the group $A$ is preserved by $\Aut(H;F)$, denote by $\bar{A}$ the image of $A$ in $\bar{H}$, and by $\tilde{A}$
the preimage of $\bar{A}$ in $H$. In other words, $\tilde{A}$ is the subgroup
of $H$ generated by $F$ and $\gamma$. Observe that the homomorphism $H\to \bar{H}$
is $\Aut(H;F)$-equivariant. The group $\bar{A}$ is preserved by the
action of $\Aut(H;F)$ on~$\bar{H}$, because every subgroup of a cyclic group
is characteristic. Hence $\tilde{A}$ is also preserved by $\Aut(H;F)$.
Furthermore, the centralizer $C_H(F)$ is preserved by $\Aut(H;F)$. One has
$$
A\subset \tilde{A}\cap C_H(F)
$$
by construction. If $\beta$ is an element
of $\tilde{A}\cap C_H(F)$, then, multiplying it by an appropriate power of $\gamma$, we may assume
that $\beta\in F$. Thus, $\beta$ is contained in the intersection~\mbox{$F\cap C_H(F)$}. The latter group is
just the center of $F$, which is trivial by assumption.
Therefore, one has $A=\tilde{A}\cap C_H(F)$, so that $A$
is preserved by~\mbox{$\Aut(H;F)$}.
This proves assertion~(ii).
\end{proof}

The next lemma will be used to analyze the extensions of cyclic groups
by the groups~$\mathfrak{A}_4$, $\mathfrak{S}_4$, and~$\mathfrak{A}_5$.

\begin{lemma}\label{lemma:An-Sn-extension}
Let $p$ be a prime number, let $\bar{H}$ be a finite cyclic group of order coprime to $p$,
and let $F$ be a finite group with trivial center. Suppose that every element of the group
$\Out(F)$ has order at most $d$. Let $H$ be a group which fits into an exact sequence
$$
1\to F\to H\to \bar{H}\to 1.
$$
Then $H$ contains a cyclic subgroup of order coprime to $p$ and index at most~\mbox{$d\cdot |F|$}
which is preserved by $\Aut(H;F)$.
\end{lemma}

\begin{proof}
Let $\alpha$ be an element of $H$ whose image in $\bar{H}$ generates the cyclic group~$\bar{H}$.
Since the order of $\bar{H}$ is coprime to $p$, the image of the element $\alpha^p$
generates~$\bar{H}$ as well. Thus, we may replace $\alpha$ by its appropriate power and assume that
its order is coprime to $p$.

Conjugation by $\alpha$ induces an automorphism of the normal subgroup~\mbox{$F\subset H$}.
By assumption, there exists a positive integer $r\le d$ such that
conjugation by~\mbox{$\alpha'=\alpha^r$} induces an inner automorphism on $F$, i.e. coincides with
conjugation by some element $f\in F$. Set
$$
\alpha''=f^{-1}\alpha'.
$$
Then for every $g\in F$ one has
$$
\alpha''g(\alpha'')^{-1}=f^{-1}\alpha'g(\alpha')^{-1}f=f^{-1}fgf^{-1}f=g.
$$
Thus, conjugation by $\alpha''$ gives a trivial automorphism of $F$. In other words,
the element~$\alpha''$ commutes with $F$. The order of~$\alpha''$ is coprime to $p$
by Lemma~\ref{lemma:extension-auxiliary}(i).

Let $A$ be the cyclic group generated by $\alpha''$.
Then the order of $A$ is coprime to $p$.
Let $\bar{A}$ be the image of~$A$ in
$\bar{H}$. Then
$\bar{A}$ is a subgroup of index at most $d$ in $\bar{H}$.
Thus the index of $A$ in $H$ equals
$$
\frac{|H|}{|A|}=\frac{|\bar{H}|}{|\bar{A}|}\cdot \frac{|F|}{|F\cap A|}\le d\cdot |F|.
$$
Finally, the group $A$ is preserved by $\Aut(H;F)$
according to Lemma~\ref{lemma:extension-auxiliary}(ii).
\end{proof}

\begin{corollary}\label{corollary:An-Sn-extension}
Let $p$ be a prime number, let $\bar{H}$ be a finite cyclic group of order coprime to $p$,
and let $F$ be one of the groups
$\mathfrak{A}_4$, $\mathfrak{S}_4$, or $\mathfrak{A}_5$.
Let $H$ be a group which fits into an exact sequence
$$
1\to F\to H\to \bar{H}\to 1.
$$
Then $H$ contains a cyclic subgroup of order coprime to $p$ and index at most~\mbox{$J\cdot |F_{(p)}|^3$}
which is preserved by $\Aut(H;F)$, where
$$
J=\begin{cases}
120, &\text{if\ } p\ge 7,\\
48, &\text{if\ } p=5,\\
\frac{40}{9}, &\text{if\ } p=3,\\
2, &\text{if\ } p=2.
\end{cases}
$$
\end{corollary}

\begin{proof}
According to Corollary~\ref{corollary:Out-PGL}(ii)
every element of the outer automorphism group $\Out(F)$ has order at most~$2$.
Also, the center of $F$ is trivial. Therefore, by Lemma~\ref{lemma:An-Sn-extension}
the group $H$ contains a cyclic subgroup of order coprime to $p$ and index at most~\mbox{$2\cdot |F|$} which is preserved by $\Aut(H;F)$.
On the other hand, we have $2\cdot |F|\le J$, cf. Examples~\ref{example:A4},
\ref{example:S4}, and~\ref{example:A5}.
\end{proof}

The following ugly lemma will be used to analyze the extensions of cyclic groups
by the groups~\mbox{$\PSL_2(\mathbf{F}_{p^k})$}
and $\PGL_2(\mathbf{F}_{p^k})$.

\begin{lemma}\label{lemma:PSL-group-extension}
Let $p$ be a prime number, and let $F$ be a finite group with trivial
center. Suppose that
there exist an element $\delta\in F$ of order coprime to~$p$ and automorphisms $\sigma, \tau\in\Aut(F)$ such that the following properties hold:
\begin{itemize}
\item
every element of $\Aut(F)$ can be decomposed
as $f\sigma^s\tau^t$ for some element~\mbox{$f\in F$} acting by inner automorphism on~$F$ and some integers~$s$ and~$t$;

\item the automorphism $\sigma$ fixes $\delta$;

\item  the automorphism $\tau$ preserves the cyclic group generated by $\delta$;

\item one has $\tau\sigma\tau^{-1}=\sigma^p$;

\item if for some integer $r$ the automorphism $\tau^r$ fixes $\delta$, then
$\tau^r$ is the trivial automorphism of $F$;

\item if $p=2$, then $\sigma$ is an inner automorphism of $F$,
and if $p\ge 3$, then $\sigma^2$ is an inner automorphism of $F$.
\end{itemize}
Let $\bar{H}$ be a finite cyclic group of order coprime to $p$.
Let~$H$ be a group which fits into an exact sequence
$$
1\to F\to H\to \bar{H}\to 1.
$$
Suppose that for every element
$\lambda\in F$ of order coprime to $p$, and every element~\mbox{$\alpha\in H$}
which normalizes the cyclic group generated by $\lambda$, the element $\alpha^2$
commutes with~$\lambda$.
Then~$H$ contains a cyclic subgroup of order coprime to $p$ and index at
most~$2\cdot |F|$ which is preserved by $\Aut(H;F)$.
\end{lemma}

\begin{proof}
Let $\alpha$ be an element of $H$ whose image in $\bar{H}$ generates~$\bar{H}$.
Since $|\bar{H}|$ is coprime to $p$, we may replace $\alpha$ by its appropriate power and assume that its order is coprime to $p$.

Conjugation by $\alpha$ induces an automorphism of the normal subgroup $F$.
By assumption on the structure of the automorphism group $\Aut(F)$,
this automorphism is a composition of conjugation by some $f\in F$ with an automorphism
of the form $\sigma^s\tau^t$.
Thus, conjugation by the element
$$
\alpha'=f^{-1}\alpha
$$
induces the automorphism $\sigma^s\tau^t$ of $F$.
In particular, conjugation by $\alpha'$ preserves the cyclic subgroup
generated by $\delta$. Hence, by the last assumption of the lemma the element
$$
\alpha''=(\alpha')^2
$$
commutes with $\delta$. Observe that conjugation by $\alpha''$
induces the automorphism
$$
\varphi=\sigma^s\tau^t\sigma^s\tau^t=\sigma^{s(1+p^t)}\tau^{2t}
$$
of $F$. Our assumptions imply that
$$
\delta=\varphi(\delta)=\sigma^{s(1+p^t)}\tau^{2t}(\delta)=\tau^{2t}(\delta).
$$
Hence $\tau^{2t}$ is the trivial automorphism of $F$, and
$\varphi=\sigma^{s(1+p^t)}$. Note that if $p$ is odd, then
$s(1+p^t)$ is even. Thus, for every $p$ the automorphism $\varphi$ is
an inner automorphism of $F$, i.e. a conjugation by some element $g\in F$.
Set
$$
\alpha'''=g^{-1}\alpha''.
$$
Then $\alpha'''$ commutes with $F$.
The order of $\alpha'''$ is coprime to $p$ by Lemma~\ref{lemma:extension-auxiliary}(i).

Let $A$ be the cyclic group generated by $\alpha'''$.
Then the order of $A$ is coprime to $p$.
Let $\bar{A}$ be the image of~$A$ in
$\bar{H}$. Then
$\bar{A}$ is a subgroup of index at most $2$ in $\bar{H}$.
Thus the index of $A$ in $H$ equals
$$
\frac{|H|}{|A|}=\frac{|\bar{H}|}{|\bar{A}|}\cdot \frac{|F|}{|F\cap A|}\le 2\cdot |F|.
$$
Finally, the group $A$ is preserved by $\Aut(H;F)$
according to Lemma~\ref{lemma:extension-auxiliary}(ii).
\end{proof}

\begin{corollary}\label{corollary:PSL-group-extension}
Let $p$ be a prime number, and let $F$ be one of the groups~\mbox{$\PSL_2(\mathbf{F}_{p^k})$}
or $\PGL_2(\mathbf{F}_{p^k})$, where $k$ is a positive integer.
Let $\bar{H}$ be a finite cyclic group of order coprime to $p$.
Let~$H$ be a group which fits into an exact sequence
$$
1\to F\to H\to \bar{H}\to 1.
$$
Suppose that for every element
$\lambda\in F$ of order coprime to $p$, and every element~\mbox{$\alpha\in H$}
which normalizes the cyclic group generated by $\lambda$, the element $\alpha^2$
commutes with~$\lambda$.
Then~$H$ contains a cyclic subgroup of order coprime to $p$ and index at
most~$2\cdot |F_{(p)}|^3$ which is preserved by $\Aut(H;F)$.
\end{corollary}

\begin{proof}
According to Theorem~\ref{theorem:PSL-PGL-basic}(iv),
the center of~$F$ is trivial.
By Proposition~\ref{proposition:Aut-PSL-PGL}
there is an isomorphism
$$
\Aut(F)\cong
\PGL_2(\mathbf{F}_{p^k})\rtimes T,
$$
where $\PGL_2(\mathbf{F}_{p^k})$ acts on $F$ by conjugations, and the
group~\mbox{$T\cong\mumu_k$} is generated by the Frobenius
automorphism $\tau$ of the field~$\mathbf{F}_{p^k}$.

Let $\sigma\in\PGL_2(\mathbf{F}_{p^k})$ be
the diagonal matrix with entries
$\zeta$ and $1$, where~$\zeta$ is a generator of the
multiplicative group of the field~$\mathbf{F}_{p^k}$, and denote
the inner automorphism of $F$ induced by this matrix also by $\sigma$.
Since the group~$\PGL_2(\mathbf{F}_{p^k})$ is generated by $\PSL_2(\mathbf{F}_{p^k})$
and $\sigma$, we see that every element of~$\Aut(F)$ can be decomposed
as $f\sigma^s\tau^t$ for some element~\mbox{$f\in F$} acting by inner automorphism on~$F$ and some integers $s$ and $t$.
One has
$$
\tau\sigma\tau^{-1}=\sigma^p.
$$
If $F=\PGL_2(\mathbf{F}_{p^k})$, then $\sigma$ is an inner automorphism of $F$.
In particular, this holds for $p=2$, because
$\PGL_2(\mathbf{F}_{2^k})=\PSL_2(\mathbf{F}_{2^k})$.
If $p\ge 3$ and $F=\PSL_2(\mathbf{F}_{p^k})$, then $F$ has index~$2$ in $\PGL_2(\mathbf{F}_{p^k})$, and so $\sigma^2$ is an inner automorphism of $F$.

Let $\delta\in\PSL_2(\mathbf{F}_{p^k})$ be
the diagonal matrix with entries
$\zeta$ and $\zeta^{-1}$.
Then the automorphism $\sigma$ fixes $\delta$, and
the automorphism $\tau$ preserves the cyclic group generated by $\delta$.
Furthermore, if for some integer $r$ the automorphism $\tau^r$ fixes~$\delta$, then
$\tau^r$ fixes $\zeta$, which means that $\tau^r$ is the trivial automorphism of
the field~$\mathbf{F}_{p^k}$ and the group $F$.

We see that the assumptions of Lemma~\ref{lemma:PSL-group-extension}
are satisfied.
Therefore, by Lemma~\ref{lemma:PSL-group-extension}
the group $H$ contains a cyclic subgroup of order coprime to $p$ and index at most~\mbox{$2\cdot |F|$} which is preserved by $\Aut(H;F)$.
On the other hand, we have
$$
|F|<|F_{(p)}|^3,
$$
see Examples~\ref{example:PSL}
and~\ref{example:PGL}.
\end{proof}

Finally, we consider extensions of cyclic groups by the groups
described in Corollary~\ref{corollary:Darafsheh}.

\begin{lemma}\label{lemma:Darafsheh-group-extension}
Let $p$ be a prime number, and let $F\cong\mumu_p^m\rtimes L$,
where~\mbox{$L\cong \mumu_n$} with $n$ coprime to $p$.
Let $\bar{H}$ be a finite cyclic group of order coprime to $p$.
Let~$H$ be a group which fits into an exact sequence
$$
1\to F\to H\to \bar{H}\to 1.
$$
Suppose that for every element
$\lambda\in L$ and every element $\alpha\in H$
which normalizes the cyclic group generated by $\lambda$, the element $\alpha^2$
commutes with~$\lambda$.
Then~$H$ contains an abelian subgroup of order coprime to $p$ and index at
most~$2\cdot |F_{(p)}|^3$ which is preserved by $\Aut(H;F)$.
\end{lemma}

\begin{proof}
Let $\alpha$ be an element of $H$ whose image in $\bar{H}$ generates~$\bar{H}$.
Since $|\bar{H}|$ is coprime to $p$, we may replace $\alpha$ by its appropriate power and assume that its order is coprime to $p$.

Set $C=C_F(F_{(p)})$ and $L'=L\cap C$.
Then~$L'$ is characteristic in $F$, and its index in $F$ does not exceed
$|F_{(p)}|^2$, see Corollary~\ref{corollary:Darafsheh}(iii).
Furthermore, the group~$L'$ is cyclic, and its order is coprime to $p$.
By assumption, the element~$\alpha^2$ commutes with $L'$.
Furthermore, conjugation by $\alpha$ preserves
the subgroup~\mbox{$F_{(p)}\cong\mumu_p^m$}, because this subgroup is
characteristic in $F$. By Lemma~\ref{lemma:Darafsheh},
there exists a positive integer $t<p^m$ such that the element
$\alpha^t$ commutes with~$F_{(p)}$. Thus, the element $\alpha'=\alpha^{2t}$
commutes with $C$, because $C$ is generated by~$L'$ and~$F_{(p)}$, see Corollary~\ref{corollary:Darafsheh}(i).
Also, we know from Corollary~\ref{corollary:Darafsheh}(i) that
$C$ is a characteristic subgroup of $F$.

Let $A\subset H$ be the subgroup generated by $L'$ and $\alpha'$.
Then $A$ is abelian, and its order is coprime to $p$.
Let $\bar{A}$ be the image of~$A$ in
$\bar{H}$. Then
$\bar{A}$ is a subgroup of index at most $2t$
in $\bar{H}$.
Thus the index of $A$ in $H$ equals
$$
\frac{|H|}{|A|}=\frac{|\bar{H}|}{|\bar{A}|}\cdot \frac{|F|}{|F\cap A|}\le 2t\cdot \frac{|F|}{|L'|}\le 2\cdot |F_{(p)}|\cdot |F_{(p)}|^2=2\cdot |F_{(p)}|^3.
$$

It remains to show that $A$ is preserved by $\Aut(H;F)$.
Let $\tilde{A}\subset H$ denote the preimage of $\bar{A}$. Then
$\tilde{A}$ is preserved by $\Aut(H;F)$.
Furthermore, the centralizer~$C_H(C)$ is preserved by $\Aut(H;F)$,
because $C$ is a characteristic
subgroup of~$F$. Denote by $\Sigma\subset C_H(C)$ the \emph{subset}
of elements of order coprime to $p$ (we do not know at the moment whether $\Sigma$
is a subgroup of $C_H(C)$ or not).
Then $\Sigma$ is also preserved by $\Aut(H;F)$.

We claim that
$$
A=\tilde{A}\cap \Sigma.
$$
Indeed, one has $A\subset \tilde{A}\cap \Sigma$
by construction.
On the other hand, if $\beta$ is an element
of $\tilde{A}\cap \Sigma$, then for an appropriate non-negative integer $k$ the element
$$
\theta=\beta\cdot (\alpha')^k
$$
is contained in $F$. Since $\theta$ and $\alpha'$ are contained in $C_H(C)$, one has
$$
\theta \in F\cap C_H(C)=C_F(C).
$$
The latter intersection coincides with $C$ by Corollary~\ref{corollary:Darafsheh}(ii).
In particular, $\alpha'$ commutes with $\theta$ and
$$
\beta=\theta\cdot (\alpha')^{-k}.
$$
Since the orders of $\alpha'$ and  $\beta$
are coprime to $p$, we conclude that the order of~$\theta$
is coprime to $p$ as well. This means that
$\theta\in L'\subset A$ by Corollary~\ref{corollary:Darafsheh}(i).
Hence~\mbox{$\beta\in A$},
which implies that $A=\tilde{A}\cap \Sigma$. Therefore, $A$ is preserved by~$\Aut(H;F)$.
\end{proof}

\section{Projective line}
\label{section:P1}

In this section we bound the indices of normal abelian subgroups of finite
groups acting on~$\PP^1$ and~$\PP^1\times\PP^1$.

\begin{theorem}[{see e.g. \cite[Chapter~XII]{Dickson}, or \cite[Theorem~2.1]{King},
or~\mbox{\cite[Theorem~2.1]{DD}}}]
\label{theorem:ADE}
Let $\Bbbk$ be a field of characteristic~\mbox{$p>0$}, and let
$$
G\subset\PGL_2(\Bbbk)
$$
be a finite group.
Then $G$ is isomorphic to one of the following groups:
\begin{itemize}
\item[(1)] a dihedral group $\DD_{2n}$, where $n\ge 2$ is coprime to~$p$;

\item[(2)] one of the groups $\mathfrak{A}_4$, $\mathfrak{S}_4$, or
$\mathfrak{A}_5$;

\item[(3)] the group $\PSL_2(\mathbf{F}_{p^k})$ for some $k\ge 1$;

\item[(4)] the group $\PGL_2(\mathbf{F}_{p^k})$ for some $k\ge 1$;

\item[(5)] a group of the form $\mumu_p^m\rtimes \mumu_n$, where $n\ge 1$
is coprime to~$p$.
\end{itemize}
\end{theorem}

\begin{remark}
If $G\subset\PGL_2(\mathbb{C})$ is a finite group, and $p$ is a prime number which
is coprime with the order of $G$, then $G$ is contained in the group $\PGL_2(\Bbbk)$
for an algebraically closed field of characteristic~$p$: the embedding can be constructed using
the same matrices as in the case of $\PGL_2(\mathbb{C})$. However,
in certain cases there exists an embedding of $G$ into $\PGL_2(\Bbbk)$ even if the characteristic of $\Bbbk$ divides~$|G|$.
For instance, there are isomorphisms
$$
\mathfrak{A}_5\cong\PSL_2(\mathbf{F}_4)\cong \PSL_2(\mathbf{F}_5)
$$
and $\mathfrak{A}_6\cong\PSL_2(\mathbf{F}_9)$, see
Theorem~\ref{theorem:PSL-PGL-basic}(ii).
Therefore, $\mathfrak{A}_5$ admits an embedding into~\mbox{$\PGL_2(\Bbbk)$} for an algebraically closed field $\Bbbk$ of \emph{any} characteristic.
\end{remark}

The next lemma is a consequence of the classification provided
by Theorem~\ref{theorem:ADE}. It gives a more precise version
of~\mbox{\cite[Lemma~8.5]{ChenShramov}}.

\begin{lemma}\label{lemma:P1}
Let $\Bbbk$ be a field of characteristic $p>0$,
and let
$$
G\subset \Aut(\PP^1)\cong \PGL_2(\Bbbk)
$$
be a finite group. Then $G$
contains a characteristic cyclic subgroup
of order coprime
to $p$ and index at most~\mbox{$J_p(\PP^1)\cdot |G_{(p)}|^{3}$},
where
\begin{equation}\label{eq:constant-PGL2}
J_p(\PP^1)=\begin{cases}
60, &\text{if\ } p\ge 7,\\
24, &\text{if\ } p=5,\\
4, &\text{if\ } p=3,\\
1, &\text{if\ } p=2.
\end{cases}
\end{equation}
Moreover, one can take $J_5(\PP^1)=2$ unless
$G$ is one of the groups~$\mathfrak{S}_4$, $\mathfrak{A}_4$, or~$\mumu_2^2$, in which case
$G$ contains a characteristic abelian subgroup of order~$4$ and index at
most~\mbox{$6=6\cdot |G_{(5)}|^3$};
and one can take $J_3(\PP^1)=\frac{20}{9}$ unless
$G\cong\mumu_2^2$, in which case $G$ is abelian itself.
\end{lemma}

\begin{proof}
The group $G$ is of one of
types (1)--(5) in the notation of Theorem~\ref{theorem:ADE}.

Suppose that $G$ is a dihedral group of order $2n$, where $n$ is coprime to~$p$.
If~\mbox{$n=2$} (so that $p\ge 3$), then $G\cong \mumu_2^2$, and the index
of the trivial group in~$G$ can be written as
$|G|=J\cdot |G_{(p)}|^3$ for $J=4$.
If $n\ge 3$, then $G$ contains a characteristic cyclic subgroup of order $n$,
see Lemma~\ref{lemma:dihedral-group-basic}. Its index in $G$ is
$$
2\le J\cdot |G_{(p)}|^3,
$$
where one can take
$$
J=\begin{cases}
2, &\text{if\ } p\ge 3,\\
1, &\text{if\ } p=2.
\end{cases}
$$

If $G$ is one of the groups $\mathfrak{A}_4$, $\mathfrak{S}_4$, or
$\mathfrak{A}_5$, then according to Examples~\ref{example:A4},
\ref{example:S4}, and~\ref{example:A5}
the index of the trivial subgroup in $G$ does not exceed
$J\cdot |G_{(p)}|^3$, where
$$
J=\begin{cases}
60, &\text{if\ } p\ge 7,\\
24, &\text{if\ } p=5,\\
\frac{20}{9}, &\text{if\ } p=3,\\
\frac{15}{16}, &\text{if\ } p=2.
\end{cases}
$$

If $G$ is one of the groups $\PSL_2(\mathbf{F}_{p^k})$ or
$\PGL_2(\mathbf{F}_{p^k})$, then the index of the trivial subgroup in $G$
does not exceed~\mbox{$J\cdot |G_{(p)}|^3$}, where
one can take $J=1$, see Examples~\ref{example:PSL}
and~\ref{example:PGL}.

If $G$ is a group of type~(5) in the notation of Theorem~\ref{theorem:ADE},
then $G$ contains a characteristic cyclic subgroup of
order coprime to $p$ and index at most $J\cdot |G_{(p)}|^2$,
where one can take $J=1$, see Corollary~\ref{corollary:Darafsheh}(iii).

In all the above cases, we see that
$J\le J_p(\PP^1)$, where $J_p(\PP^1)$ is given by~\eqref{eq:constant-PGL2}.
In other words, the group $G$ always contains a characteristic cyclic subgroup
of order coprime
to~$p$ and index at most~\mbox{$J_p(\PP^1)\cdot |G_{(p)}|^{3}$}.
This proves the first assertion of the lemma.

Finally, if $p=5$, then
we see from the above computations that
either $G$ contains a characteristic cyclic subgroup of order coprime to $5$ and index
at most~\mbox{$2\cdot |G_{(5)}|^3$},
or~$G$ is isomorphic to one of the groups $\mathfrak{S}_4$, $\mathfrak{A}_4$, or~$\mumu_2^2$.
Each of the latter groups contains a characteristic abelian subgroup $V_4\cong\mumu_2^2$.
If $p=3$, then the above computations show that
either $G$ contains a characteristic cyclic subgroup of order coprime to $3$ and index
at most~\mbox{$\frac{20}{9}\cdot |G_{(3)}|^3$},
or $G\cong\mumu_2^2$. This gives the second assertion of the lemma.
\end{proof}

\begin{corollary}\label{corollary:P1xP1}
Let $\Bbbk$ be a field of characteristic $p>0$, and let
$$
\hat{\Gamma}=\big(\PGL_2(\Bbbk)\times\PGL_2(\Bbbk)\big)\rtimes\mumu_2,
$$
where the non-trivial element of $\mumu_2$ acts by interchanging the factors.
Then every finite subgroup $G$ of $\hat{\Gamma}$ contains
a normal abelian subgroup of order coprime to $p$
and index at most~\mbox{$J_p(\PP^1\times\PP^1)\cdot |G_{(p)}|^{3}$},
where
\begin{equation}\label{eq:J-P1xP1}
J_p(\PP^1\times\PP^1)=\begin{cases}
7200, &\text{if\ } p\ge 7,\\
72, &\text{if\ } p=5,\\
10, &\text{if\ } p=3,\\
1, &\text{if\ } p=2.
\end{cases}
\end{equation}
\end{corollary}

\begin{proof}
We know from Lemma~\ref{lemma:P1} that every finite group $F\subset\PGL_2(\Bbbk)$
contains a characteristic abelian subgroup of order coprime to $p$ and index
at most~\mbox{$J\cdot |F_{(p)}|^3$}, where
$J=J_p(\PP^1)$ is given by~\eqref{eq:constant-PGL2} for $p\ge 7$ and $p=2$,
while~\mbox{$J=6$} for~\mbox{$p=5$}, and $J=\frac{20}{9}$ for $p=3$.
Thus, if $G\subset \PGL_2(\Bbbk)\times\PGL_2(\Bbbk)$, then the required
assertion follows from
Lemma~\ref{lemma:product}.

Therefore, we assume that $G$ contains a subgroup
$G'$ of index $2$ such that
$$
G'\subset \PGL_2(\Bbbk)\times\PGL_2(\Bbbk),
$$
and an element of $G\setminus G'$ acts on $\PGL_2(\Bbbk)\times\PGL_2(\Bbbk)$ by a composition
of conjugation by some element of $G'$ with the non-trivial permutation of the factors.
Denote by
$$
\pi_1, \pi_2\colon G'\to \PGL_2(\Bbbk)
$$
the projections to the two factors.
Thus, $\pi_1(G')$ and $\pi_2(G')$ are isomorphic finite subgroups of $\PGL_2(\Bbbk)$,
and $G'\subset \pi_1(G')\times\pi_2(G')$. Observe that
if $N$ is a normal abelian subgroup in $\pi_1(G')\cong\pi_2(G')$
then the intersection $A=G'\cap (N\times N)$ is a normal abelian subgroup in $G$
whose index in $G$ is
$$
\frac{|G|}{|A|}=2\cdot\frac{|G'|}{|A|}\le 2\cdot \frac{|\pi_1(G')\times\pi_2(G')|}{|N\times N|}=
2\cdot \left(\frac{|\pi_1(G')|}{|N|}\right)^2.
$$
We check the possibilities for $\pi_1(G')\cong\pi_2(G')$
according to the classification of finite subgroups of $\PGL_2(\Bbbk)$
provided by Theorem~\ref{theorem:ADE}.

Suppose that $\pi_1(G')\cong\pi_2(G')\cong\DD_{2n}$ is the dihedral group of order $2n$, where~$n$ is  coprime to $p$. Then $\DD_{2n}$ contains a normal cyclic subgroup $C$ of order~$n$ and index $2$.
Thus, the intersection $A=G'\cap (C\times C)$
is a normal abelian subgroup in $G$ of order coprime to $p$
and index at most $2\cdot 2^2=8$. This does not exceed~$J_p(\PP^1\times\PP^1)$ if $p\ge 3$.
If $p=2$, we have
$$
8=|\pi_1(G')_{(2)}|^3\le |G'_{(2)}|^3<|G_{(2)}|^3.
$$

Suppose that $\pi_1(G')\cong\pi_2(G')$ is one of the groups $\mathfrak{A}_4$ or $\mathfrak{S}_4$.
Let $V_4$ be the normal abelian subgroup of order $4$ in $\pi_1(G')$.
If $p\neq 2$, then the intersection
$$
A=G'\cap (V_4\times V_4)
$$
is a normal abelian
subgroup in $G$ of order coprime to $p$ and index at most~\mbox{$2\cdot 6^2=72$}.
This does not exceed $J_p(\PP^1\times\PP^1)$ if $p\ge 5$; if $p=3$, we have
$$
72<10\cdot 27=10\cdot |\pi_1(G')_{(3)}|^3\le 10\cdot |G'_{(3)}|^3=10\cdot |G_{(3)}|^3.
$$
If $p=2$ and $\pi_1(G')\cong\mathfrak{A}_4$, then the index of the trivial group in $G$ is
$$
|G|\le 2\cdot 12^2=288<8^3=(2\cdot |\pi_1(G')_{(2)}|)^3\le (2\cdot |G'_{(2)}|)^3=|G_{(2)}|^3.
$$
If $p=2$ and $\pi_1(G')\cong\mathfrak{S}_4$, then the index of the trivial group in $G$ is
$$
|G|\le 2\cdot 24^2=1152<16^3=(2\cdot |\pi_1(G')_{(2)}|)^3\le (2\cdot |G'_{(2)}|)^3=|G_{(2)}|^3.
$$

Suppose that $\pi_1(G')\cong\pi_2(G')\cong\mathfrak{A}_5$.
The kernel $H$ of the projection $\pi_1$ is a normal subgroup of $G'$, and $\pi_2$ maps this subgroup
isomorphically to a normal subgroup of $\mathfrak{A}_5$.
Hence either $H$ is trivial, or isomorphic to $\mathfrak{A}_5$. In the former case
$G'\cong\mathfrak{A}_5$ and $|G|=120$; thus, the trivial subgroup of $G$ has index $120$,
which does not exceed $J_p(\PP^1\times\PP^1)\cdot |G_{(p)}|^3$.
In the latter case $G'\cong\mathfrak{A}_5\times\mathfrak{A}_5$ and~\mbox{$|G|=7200$};
thus, the trivial subgroup of $G$ has index $7200$,
which again does not exceed $J_p(\PP^1\times\PP^1)\cdot |G_{(p)}|^3$.

Suppose that $\pi_1(G')\cong\pi_2(G')\cong\PSL_2(\mathbf{F}_{p^k})$ for some $k\ge 1$,
where either~\mbox{$k\ge 2$}, or $p\ge 5$. Then the group $\PSL_2(\mathbf{F}_{p^k})$ is simple, see Theorem~\ref{theorem:PSL-PGL-basic}(i).
The kernel $H$ of the projection $\pi_1$ is a normal subgroup of $G'$, and $\pi_2$ maps this subgroup
isomorphically to a normal subgroup of $\PSL_2(\mathbf{F}_{p^k})$.
Hence $H$ is either trivial, or isomorphic to $\PSL_2(\mathbf{F}_{p^k})$.
If $H$ is trivial, then $G'\cong\PSL_2(\mathbf{F}_{p^k})$.
Thus, if $p\ge 3$, the trivial subgroup of $G$ has index
$$
|G|=2\cdot |\PSL_2(\mathbf{F}_{p^k})|=2\cdot \frac{|\SL_2(\mathbf{F}_{p^k})|}{2}=p^k(p^{2k}-1)
<p^{3k}=|G_{(p)}|^3,
$$
and if $p=2$, it has
index
$$
|G|=2\cdot |\PSL_2(\mathbf{F}_{2^k})|=2\cdot|\SL_2(\mathbf{F}_{2^k})|=2^{k+1}(2^{2k}-1)<2^{3(k+1)}=
|G_{(2)}|^3,
$$
cf. Example~\ref{example:PSL}.
If $H\cong \PSL_2(\mathbf{F}_{p^k})$, then $G'\cong \PSL_2(\mathbf{F}_{p^k})\times\PSL_2(\mathbf{F}_{p^k})$.
Thus, if $p\ge 3$, the trivial subgroup of $G$ has index
$$
|G|=2\cdot |\PSL_2(\mathbf{F}_{p^k})|^2=\frac{p^{2k}(p^{2k}-1)^2}{2}
<p^{6k}=|G_{(p)}|^3,
$$
and if $p=2$, it has
index
$$
|G|=2\cdot |\PSL_2(\mathbf{F}_{2^k})|^2=2^{2k+1}(2^{2k}-1)^2<2^{3(2k+1)}=
|G_{(2)}|^3.
$$

Suppose that $\pi_1(G')\cong\pi_2(G')\cong\PSL_2(\mathbf{F}_{p})$,
where $p\le 3$. Then $\pi_1(G')\cong \mathfrak{A}_4$ if $p=3$, and
$\pi_1(G')\cong \mathfrak{S}_3$ if $p=2$, see
Theorem~\ref{theorem:PSL-PGL-basic}(ii).
The former case was already considered above.
Thus, suppose that the latter case takes place.
If the kernel $H\subset G'$ of $\pi_1$ does not contain elements of order~$2$,
the trivial subgroup of $G$ has index
$$
|G|=2\cdot |G'|\le 2\cdot 3\cdot |\mathfrak{S}_3|=36<4^3=|G_{(2)}|^3.
$$
If $H$ contains elements of order $2$, then~\mbox{$H\cong\mathfrak{S}_3$}
and~\mbox{$G'\cong \mathfrak{S}_3\times\mathfrak{S}_3$}; thus
the trivial subgroup of $G$ has index
$$
|G|=72<8^3=|G_{(2)}|^3.
$$

Suppose that $\pi_1(G')\cong\pi_2(G')\cong\PGL_2(\mathbf{F}_{p^k})$ for some $k\ge 1$. Suppose also that either~\mbox{$k\ge 2$}, or $p\ge 5$.
Let $H\subset G'$ be the kernel of the projection~$\pi_1$, and set
$$
H_2=\pi_2(H)\cap \PSL_2(\mathbf{F}_{p^k})\subset\PGL_2(\mathbf{F}_{p^k}).
$$
Then $H_2$ is a normal subgroup of $\PSL_2(\mathbf{F}_{p^k})$. Since the latter group is simple, we see that either $H_2$ is trivial, or $H_2\cong \PSL_2(\mathbf{F}_{p^k})$.
If $H_2$ is trivial, then
$$
|H|\le \frac{|\PGL_2(\mathbf{F}_{p^k})|}{|\PSL_2(\mathbf{F}_{p^k})|}\le 2,
$$
and $|G'|\le 2\cdot |\PGL_2(\mathbf{F}_{p^k})|$.
Thus, if $p\ge 3$, the trivial subgroup of $G$ has index
$$
|G|=2\cdot |G'|\le 4\cdot |\PGL_2(\mathbf{F}_{p^k})|= 4p^{k}(p^{2k}-1)\le 4p^{3k}=4|\pi_1(G')_{(p)}|^3=
4|G_{(p)}|^3;
$$
if $p=2$, the trivial group has index
$$
|G|\le 4\cdot 2^{3k}=2^{3k+2}<2^{3(k+1)}=(2\cdot |\pi_1(G')_{(2)}|)^3=|G_{(2)}|^3.
$$
If $H_2\cong\PSL_2(\mathbf{F}_{p^k})$, then $|G'_{(p)}|=p^{2k}$, and the trivial subgroup of $G$ has index
$$
|G|\le 2\cdot |\PGL_2(\mathbf{F}_{p^k})|^2=2p^{2k}(p^{2k}-1)^2<2p^{6k}.
$$
If $p\ge 3$, then
$$
2p^{6k}=2\cdot |G'_{(p)}|^3=2\cdot |G_{(p)}|^3.
$$
If $p=2$, then
$$
2\cdot 2^{6k}<2^{3(2k+1)}=(2\cdot |G'_{(2)}|)^3=|G_{(2)}|^3.
$$

Suppose that $\pi_1(G')\cong\pi_2(G')\cong\PGL_2(\mathbf{F}_{p})$,
where $p\le 3$. Then $\pi_1(G')\cong \mathfrak{S}_4$ if $p=3$, and
$\pi_1(G')\cong \mathfrak{S}_3$ if $p=2$,
see Theorem~\ref{theorem:PSL-PGL-basic}(iii).
Both of these cases were already considered above.

Finally, if $R\cong\pi_1(G')\cong\pi_2(G')$
is a group of type~(5) in the notation of Theorem~\ref{theorem:ADE},
then the required assertion follows from Corollary~\ref{corollary:Darafsheh-product}.
\end{proof}

\section{Projective plane}
\label{section:P2}

In this section we give preliminary bounds for the indices of normal abelian subgroups of finite
groups acting on~$\PP^2$ over fields of odd characteristic.
Unfortunately, in this case a classification as fine as one provided by
Theorem~\ref{theorem:ADE} is not available, but in odd characteristic
one can describe finite subgroups of~\mbox{$\PGL_3(\Bbbk)$}
in terms of geometric configurations up to several explicit exceptions.
We will say that a subgroup $G\subset\PGL_3(\Bbbk)$ \emph{preserves a triangle}, if
it permutes three non-collinear points of $\PP^2$ defined over an algebraic closure of~$\Bbbk$.
Similarly, the group $G$ \emph{preserves a smooth conic}
if there exists a smooth $G$-invariant conic on~$\PP^2$ over the algebraic closure
of~$\Bbbk$.

\begin{theorem}[{see \cite{Mitchell} or \cite[Theorem~2.4]{King}}]
\label{theorem:Mitchell}
Let $p\neq 2$ be a prime number.
Let
$$
G\subset \PSL_3(\mathbf{F}_{p^n})
$$
be a subgroup.
Then either $G$ preserves a point, or a line, or a triangle, or a smooth conic, or~$G$ is isomorphic to one of the following groups:
\begin{itemize}
\item[(1)] $\PSL_3(\mathbf{F}_{p^k})$ for some $k\le n$;

\item[(2)] $\PGL_3(\mathbf{F}_{p^k})$ for some $k<n$;

\item[(3)] $\PSU_3(\mathbf{F}_{p^k})$ for some $k\le \frac{n}{2}$;

\item[(4)] $\PU_3(\mathbf{F}_{p^k})$  for some $k<\frac{n}{2}$;

\item[(5)] the Hessian group $\mumu_3^2\rtimes\SL_2(\mathbf{F}_3)$ of order $216$, or one of its subgroups containing~$\mumu_3^2$ (these cases are possible only if $p\neq 3)$;

\item[(6)] $\PSL_2(\mathbf{F}_7)$;

\item[(7)] $\mathfrak{A}_6$, or a group containing $\mathfrak{A}_6$ as a subgroup of index~$2$
(both possible only if $p=5$);

\item[(8)] $\mathfrak{A}_7$ (possible only if $p=5$).
\end{itemize}
\end{theorem}

\begin{remark}
Theorem~\ref{theorem:Mitchell} does not claim that the groups from the list are indeed realized
as subgroups of $\PSL_3(\mathbf{F}_{p^n})$ for any given $p$ and $n$. For more detailed information
concerning this question we refer the reader to~\cite{Mitchell} and~\cite{King}.
\end{remark}

\begin{lemma}\label{lemma:triangle}
Let $\Bbbk$ be a field of characteristic $p>0$.
Let
$$
G\subset \Aut(\PP^2)\cong \PGL_3(\Bbbk)
$$
be a finite group.
Suppose that $G$ preserves a triangle.
Then $G$ contains a normal abelian subgroup
of order coprime
to $p$ and index at most~\mbox{$J\cdot |G_{(p)}|^{3}$},
where
\begin{equation*}
J=\begin{cases}
6, &\text{if\ } p\ge 5,\\
2, &\text{if\ } p=3,\\
3, &\text{if\ } p=2.
\end{cases}
\end{equation*}
\end{lemma}

\begin{proof}
The group $G$ fits into an exact sequence
$$
1\to H\to G\to\bar{G}\to 1,
$$
where $H$ preserves three non-collinear points on $\PP^2$ over the algebraic closure~$\bar{\Bbbk}$
of~$\Bbbk$,
and $\bar{G}$ is a subgroup of $\mathfrak{S}_3$.
Therefore, we have
$$
|\bar{G}|\in\big\{1, 2, 3, 6\big\}.
$$
Hence
\begin{equation*}
\frac{|\bar{G}|}{|\bar{G}_{(p)}|^3}\le J.
\end{equation*}
On the other hand, the preimage of $H$ in $\GL_3\left(\bar{\Bbbk}\right)$
consists of the matrices which are diagonal in the basis corresponding to the three points
preserved by~$H$, which implies that $H$ is abelian.
Thus, the required assertion follows from Lemma~\ref{lemma:large-abelian}.
\end{proof}

\begin{lemma}\label{lemma:P2-preliminary}
Let $\Bbbk$ be a field of characteristic $p>2$.
Let
$$
G\subset \Aut(\PP^2)\cong \PGL_3(\Bbbk)
$$
be a finite group.
Suppose that $G$ preserves neither a point nor a line
on~$\PP^2$.
Then~$G$ contains a normal abelian subgroup
of order coprime
to $p$ and index at most~\mbox{$J\cdot |G_{(p)}|^{3}$},
where
\begin{equation*}
J=\begin{cases}
360, &\text{if\ } p\ge 7,\\
168, &\text{if\ } p=5,\\
7, &\text{if\ } p=3.\\
\end{cases}
\end{equation*}
\end{lemma}

\begin{proof}
We may assume that the field $\Bbbk$ is algebraically closed.
Thus, one has
$$
\PGL_3(\Bbbk)\cong\PSL_3(\Bbbk),
$$
and hence $G$ is contained in
the group $\PSL_3(\mathbf{F}_{p^n})$ for some positive integer $n$.

If $G$ preserves a triangle, then the required assertion holds by Lemma~\ref{lemma:triangle}.
If $G$ preserves a smooth conic, then~\mbox{$G\subset\PGL_2(\Bbbk)$},
and the assertion holds by Lemma~\ref{lemma:P1}.
Otherwise, $G$ is a group of one of
types~\mbox{(1)--(8)} in the notation of Theorem~\ref{theorem:Mitchell}.

If $G$ is isomorphic to one of the groups
$\PSL_3(\mathbf{F}_{p^k})$ or
$\PGL_3(\mathbf{F}_{p^k})$, then~\mbox{$|G_{(p)}|=p^{3k}$}. Thus,
the index of the trivial subgroup in $G$ equals
$$
|G|\le |\PGL_3(\mathbf{F}_{p^k})|=p^{3k}(p^{3k}-1)(p^{2k}-1)<p^{9k}=|G_{(p)}|^3.
$$

If $G$ is isomorphic to one of the groups
$\PSU_3(\mathbf{F}_{p^k})$ or $\PU_3(\mathbf{F}_{p^k})$,
then~\mbox{$|G_{(p)}|=p^{3k}$}, while
$$
|G|\le p^{3k}(p^{3k}+1)(p^{2k}-1)(p^k+1),
$$
see e.g. \cite[\S3.6]{Wilson}. Thus,
the index of the trivial subgroup in $G$ equals
\begin{multline*}
|G|\le p^{3k}(p^{3k}+1)(p^{2k}-1)(p^k+1)=
p^{3k}(p^{6k}+p^{5k}-p^{4k}+p^{2k}-p^k-1)\\
\\ \le p^{9k}\cdot\left(1+\frac{1}{p^k}\right)\le
p^{9k}\cdot\left(1+\frac{1}{3}\right)=\frac{4}{3}\cdot |G_{(p)}|^3.
\end{multline*}

If $p\neq 3$, and $G$ is either the group $\mumu_3^2\rtimes\SL_2(\mathbf{F}_3)$, or one of
its subgroups containing~$\mumu_3^2$, then the index of the normal subgroup $\mumu_3^2$
in $G$ does not exceed
$$
\frac{|G|}{9}\le 24.
$$

If $G\cong\PSL_2(\mathbf{F}_7)$, then $|G|=168=2^3\cdot 3\cdot 7$. Thus, the index of the trivial subgroup
equals
$|G|=I \cdot |G_{(p)}|^3$, where $I$ does not exceed
$$
\begin{cases}
168, &\text{if\ } p\ge 5,\\
7, &\text{if\ } p=3.\\
\end{cases}
$$

If $p=5$, and $G$ is either isomorphic to $\mathfrak{A}_6$,
or contains $\mathfrak{A}_6$ as a subgroup of index $2$, then
$|G_{(5)}|=5$, and
the index of the trivial subgroup equals
$$
|G|\le 2\cdot |\mathfrak{A}_6|=720=\frac{720}{125}\cdot |G_{(5)}|^3<
6\cdot |G_{(5)}|^3
$$

Finally, if $p=5$ and $G\cong\mathfrak{A}_7$, then the index of the trivial subgroup equals
$$
|G|=2520=\frac{504}{25}\cdot |G_{(5)}|^3<21\cdot |G_{(5)}|^3.
$$
\end{proof}

\section{Conic bundles}
\label{section:CB}

In this section we bound the indices of normal abelian subgroups of finite
groups acting on conic bundles.

Let $S$ be a smooth geometrically irreducible projective surface,
and let~\mbox{$\phi\colon S\to C$}
be a surjective morphism to a smooth curve.
One says that~$\phi$ is a \emph{conic bundle},
if the fiber of $\phi$
over the scheme-theoretic generic point of $C$ is  smooth
and geometrically  irreducible, and the anticanonical line
bundle~$\omega^{-1}_S$ is very ample over $C$.

For the following fact, we refer the reader to~\mbox{\cite[Lemma~5.2]{Serre-2009}}
or~\mbox{\cite[Lemma~9.1]{ChenShramov}}.

\begin{lemma}\label{lemma:Serre}
Let $p$ be a prime number, and let
$\Bbbk$ be an algebraically closed field of characteristic $p$.
Let~\mbox{$H\subset\Aut(S)$} be a finite group,
and let $\phi\colon S\to C$
be an $H$-equivariant conic bundle.
Denote by $F$ the subgroup of $H$ which consists of the elements preserving every fiber of~$\phi$. Then for every element
$\lambda\in F$ of order coprime to $p$, and every element~\mbox{$\alpha\in H$}
which normalizes the cyclic group generated by $\lambda$, the element $\alpha^2$
commutes with~$\lambda$.
\end{lemma}

The next assertion is a more precise version of~\mbox{\cite[Lemma~9.2]{ChenShramov}}.

\begin{proposition}\label{proposition:CB}
Let $\Bbbk$ be an algebraically closed field of characteristic $p>0$, and let
$S$ be a smooth rational projective surface over~$\Bbbk$.
Let~\mbox{$G\subset\Aut(S)$} be a finite group, and let
$$
\phi\colon S\to C
$$
be a $G$-equivariant conic bundle.
Then $G$ contains a normal abelian subgroup of order coprime
to $p$ and index at most~\mbox{$J_p^{\mathrm{cb}}\cdot |G_{(p)}|^{3}$},
where
\begin{equation}\label{eq:J-CB}
J_p^{\mathrm{cb}}=\begin{cases}
7200, &\text{if\ } p\ge 7,\\
144, &\text{if\ } p=5,\\
\frac{800}{81}, &\text{if\ } p=3,\\
2, &\text{if\ } p=2.
\end{cases}
\end{equation}
\end{proposition}

\begin{proof}
Since $S$ is rational, one has $C\cong\PP^1$.
There is an exact sequence of groups
$$
1\to F\to G\to \bar{G}\to 1,
$$
where the action of $F$ is fiberwise
with respect to $\phi$, while $\bar{G}$ acts faithfully on~$\PP^1$.

First, suppose that either $p\not\in \{3,5\}$, or $p=5$ and
$\bar{G}\not\in\{\mathfrak{S}_4,\mathfrak{A}_4, \mumu_2^2\}$, or~\mbox{$p=3$} and $\bar{G}\not\cong\mumu_2^2$.
By Lemma~\ref{lemma:P1} there exists a normal
cyclic subgroup~$\bar{H}$ in~$\bar{G}$
of order coprime to~$p$ whose index in $\bar{G}$ does not exceed
$J\cdot |\bar{G}_{(p)}|^3$, where~\mbox{$J=J_p(\PP^1)$}
is given by~\eqref{eq:constant-PGL2} if $p\not\in\{3, 5\}$,
one has $J=2$ if~\mbox{$p=5$}, and one has~\mbox{$J=\frac{20}{9}$} if~\mbox{$p=3$}.

Let $H$ be the preimage of $\bar{H}$ in $G$, so that there
is an exact sequence of groups
$$
1\to F\to H\to \bar{H}\to 1.
$$
In particular, $H$ is a normal subgroup of~$G$,
so that conjugation by any element of~$G$ gives an automorphism of~$H$.
Moreover, since $F$ is a normal subgroup of~$G$, such automorphisms of $H$ preserve $F$;
in other words, conjugation by any element of~$G$ gives an element of the group~\mbox{$\Aut(H;F)$}.
Thus, every subgroup of~$H$ which is preserved by~\mbox{$\Aut(H;F)$} is normal in~$G$.
Observe also that according to Lemma~\ref{lemma:Serre}, for every element
$\lambda\in F$ of order coprime to $p$, and every element~\mbox{$\alpha\in H$}
which normalizes the cyclic group generated by $\lambda$, the element~$\alpha^2$
commutes with~$\lambda$.

The group $F$ is of one of
types (1)--(5) in the notation of Theorem~\ref{theorem:ADE}.
If~$F$ is of type~(1)
and~\mbox{$F\cong\mumu_2^2$}, then $p\ge 3$, and
by Lemma~\ref{lemma:small-dihedral-group-extension} the group $H$ contains an abelian subgroup $A$ of order coprime to $p$ which is normal in $G$ and has index at most~$I$ in $H$, where $I=3$. In particular, the index of $A$ in $H$ does
not exceed~\mbox{$I\cdot |F_{(p)}|^3$}.
If $F\cong\DD_{2n}$ is a dihedral group of order $2n$ with $n\ge 3$,
then by Lemma~\ref{lemma:dihedral-group-extension} the group $H$ contains an abelian subgroup $A$ of order coprime to $p$ and index at most~$4$ which is normal in $G$.
Thus, the index of $A$ in $H$ does not exceed~\mbox{$I\cdot |F_{(p)}|^3$}, where
$$
I=\begin{cases}
4, &\text{if\ } p\ge 3,\\
1, &\text{if\ } p=2.
\end{cases}
$$
If $F$ is one of the groups
$\mathfrak{A}_4$, $\mathfrak{S}_4$, or $\mathfrak{A}_5$, then
by Corollary~\ref{corollary:An-Sn-extension} the group $H$ contains a cyclic subgroup $A$ of order coprime to $p$ and index at most $I\cdot |F_{(p)}|^3$ which is normal in $G$,
where
$$
I=\begin{cases}
120, &\text{if\ } p\ge 7,\\
48, &\text{if\ } p=5,\\
\frac{40}{9}, &\text{if\ } p=3,\\
2, &\text{if\ } p=2.
\end{cases}
$$
If $F$ is either one of the groups~\mbox{$\PSL_2(\mathbf{F}_{p^k})$}
or~\mbox{$\PGL_2(\mathbf{F}_{p^k})$}, or a group of type~(5) in the notation of
Theorem~\ref{theorem:ADE}, then by Corollary~\ref{corollary:PSL-group-extension}
and Lemma~\ref{lemma:Darafsheh-group-extension} the group $H$ contains an abelian subgroup $A$ of order coprime to $p$ and index at most $I\cdot |F_{(p)}|^3$ which is normal in $G$,
where $I=2$. In any case,
we see that
$$
I\cdot J\le J_p^{\mathrm{cb}}.
$$
Therefore, $A$ is a normal abelian subgroup of $G$ of order coprime to $p$ and
index at most
$$
I\cdot J\cdot |F_{(p)}|^3\cdot |\bar{G}_{(p)}|^3\le J_p^{\mathrm{cb}}\cdot |G_{(p)}|^3.
$$

Now suppose that $p=5$ and $\bar{G}\in\{\mathfrak{S}_4,\mathfrak{A}_4,\mumu_2^2\}$.
By Lemma~\ref{lemma:P1} there exists a characteristic abelian
subgroup $A$ in $F$ of order coprime to $5$
and index at most~\mbox{$6\cdot |F_{(5)}|^3$}. Therefore, $A$ is normal in
$G$, and its index in $G$ is
$$
\frac{|G|}{|A|}=
|\bar{G}|\cdot \frac{|F|}{|A|}\le
24\cdot 6\cdot |F_{(5)}|^3=
144\cdot  |G_{(5)}|^3=
J_5^{\mathrm{cb}}\cdot |G_{(5)}|^3.
$$

Finally,
suppose that $p=3$ and $\bar{G}\cong\mumu_2^2$.
By Lemma~\ref{lemma:P1} there exists a characteristic abelian
subgroup $A$ in $F$ of order coprime to $3$
and index at most~\mbox{$\frac{20}{9}\cdot |F_{(3)}|^3$}. Therefore, $A$ is normal in
$G$, and its index in $G$ is
$$
\frac{|G|}{|A|}=
|\bar{G}|\cdot \frac{|F|}{|A|}\le
4\cdot \frac{20}{9}\cdot |F_{(3)}|^3<J_3^{\mathrm{cb}}\cdot |G_{(3)}|^3.
$$
\end{proof}

\begin{remark}
The proof of~\mbox{\cite[Lemma~9.2]{ChenShramov}} followed roughly the same lines as the above proof of Proposition~\ref{proposition:CB}.
However, the abelian subgroup constructed in the former proof was claimed to be normal,
but its normality was not properly checked. Instead of filling this gap, we decided to present here a modified version of the proof based on a case by case study
of the possible subgroups~$F$ and the auxiliary assertions
provided in Section~\ref{section:extensions}.
\end{remark}

Proposition~\ref{proposition:CB} allows us to finalize our understanding
of finite groups acting on~$\PP^2$ in odd characteristic.

\begin{corollary}\label{corollary:P2}
Let $\Bbbk$ be a field of characteristic $p>2$.
Let
$$
G\subset \Aut(\PP^2)\cong \PGL_3(\Bbbk)
$$
be a finite group.
Then~$G$ contains a normal abelian subgroup
of order coprime
to $p$ and index at most~\mbox{$J\cdot |G_{(p)}|^{3}$},
where
\begin{equation}\label{eq:constant-PGL3-odd}
J=\begin{cases}
7200, &\text{if\ } p\ge 7,\\
168, &\text{if\ } p=5,\\
\frac{800}{81}, &\text{if\ } p=3.\\
\end{cases}
\end{equation}
\end{corollary}

\begin{proof}
If $G$ preserves neither a point nor a line
on~$\PP^2$, then the required assertion holds by Lemma~\ref{lemma:P2-preliminary}.
If $G$ preserves a point $P\in \PP^2$, then $G$ acts on the blow up of $\PP^2$ at $P$, which
has a (unique) structure of a conic bundle over $\PP^1$. Therefore, $G$
contains a normal abelian subgroup
of order coprime
to $p$ and index at most~\mbox{$J_p^{\mathrm{cb}}\cdot |G_{(p)}|^{3}$},
where $J_p^{\mathrm{cb}}\le J$ is given by~\eqref{eq:J-CB}.
Finally, if $G$ preserves a line on $\PP^2$, then $G$ acts on the projectively dual plane $(\PP^2)^\vee\cong\PP^2$
with a fixed point, and the previous argument applies.
\end{proof}

\section{Del Pezzo surfaces}
\label{section:dP}

In this section we bound the indices of normal abelian subgroups of finite
groups acting on del Pezzo surfaces. Although a complete classification
of automorphism groups of del Pezzo surfaces over any algebraically closed field
is known (see~\cite{DD}, \cite{DM-odd}, and~\cite{DM-2}),
we avoid using it as much as possible. Instead, we rely only on the
most general results concerning del Pezzo surfaces and their automorphism groups
which are valid in arbitrary characteristic. In particular, we recall that
the automorphism group of any del Pezzo surface of degree~$5$, $4$, and~$3$ is a subgroup
of the Weyl group~$\mathrm{W}(\mathrm{A}_4)$, $\WD$, and~$\WE$, respectively, see
e.g.~\mbox{\cite[Corollary 8.2.40]{Dolgachev}} or~\mbox{\cite[\S2.2]{DM-odd}}.

Let us recall the following general result.

\begin{theorem}[{see \cite[Th\'eor\`eme 5.1 and Lemme 5.2]{Serre-2008}}]
\label{theorem:Serre-lift}
Let $S$ be a smooth geometrically rational projective
surface over a field of characteristic $p>0$, and let
$G\subset \Aut(S)$ be a finite group whose order is coprime to $p$.
Then there exists a smooth projective geometrically rational surface $\hat{S}$
over a field of characteristic~$0$ such that~\mbox{$G\subset\Aut(\hat{S})$}.
Furthermore, if $S$ is a del Pezzo surface, then for~$\hat{S}$ one can take a del Pezzo
surface of the same degree.
\end{theorem}

Now we consider automorphism groups of del Pezzo surfaces case by case,
depending on the degree.

\begin{lemma}\label{lemma:dP6}
Let $S$ be a del Pezzo surface of degree $6$ over an algebraically closed field
of characteristic $p>0$.
Then every group $G\subset\Aut(S)$ contains a normal abelian subgroup of order coprime to $p$ and index at most
$J\cdot |G_{(p)}|^3$, where
\begin{equation}\label{eq:dP6}
J=\begin{cases}
12, &\text{if\ } p\ge 5,\\
4, &\text{if\ } p=3,\\
3, &\text{if\ } p=2.
\end{cases}
\end{equation}
\end{lemma}

\begin{proof}
The automorphism group of $S$ fits into an exact sequence
$$
1\to (\Bbbk^*)^2\to\Aut(S)\to \mathfrak{S}_3\times\mumu_2\to 1,
$$
see for instance \cite[Theorem~8.4.2]{Dolgachev} or \cite[3.1.4]{DM-odd}.
Thus, any finite subgroup $G\subset\Aut(S)$
fits into an exact sequence
$$
1\to H\to G\to\bar{G}\to 1,
$$
where $H$ is abelian, and $\bar{G}\subset \mathfrak{S}_3\times\mumu_2$.
We have
$$
|\bar{G}|\in\big\{1,2,3,4,6,12\big\}.
$$
Hence
$$
\frac{|\bar{G}|}{|\bar{G}_{(p)}|^3}\le J.
$$
Now the assertion follows from Lemma~\ref{lemma:large-abelian}.
\end{proof}

\begin{lemma}\label{lemma:dP5}
Let $S$ be a del Pezzo surface of degree $5$ over a field
of characteristic $p>0$.
Then every group $G\subset\Aut(S)$ contains a normal abelian subgroup of order
coprime to $p$
and index at most $J\cdot |G_{(p)}|^3$, where $J$ is given by~\eqref{eq:auxiliary-subgroups}.
\end{lemma}

\begin{proof}
We know that $\Aut(S)$ is isomorphic to a subgroup of the Weyl group
$$
\mathrm{W}(\mathrm{A}_4)\cong \mathfrak{S}_5.
$$
Therefore, the required assertion follows from Lemma~\ref{lemma:auxiliary-subgroups}(i).
\end{proof}

\begin{lemma}\label{lemma:dP4-forbidden}
Let $S$ be a del Pezzo surface of degree $4$ over a field of characteristic $3$.
Then the automorphism group $\Aut(S)$ does not contain subgroups isomorphic to
$\mumu_2^4\rtimes(\mumu_5\rtimes\mumu_4)$.
\end{lemma}

\begin{proof}
Suppose that $\Aut(S)$ contains a subgroup $G\cong\mumu_2^4\rtimes(\mumu_5\rtimes\mumu_4)$.
Then there is a faithful action of $\mumu_5\rtimes\mumu_4$ on $\PP^1$, see for instance~\mbox{\cite[\S3]{DD}}.
However, the latter is impossible by Theorem~\ref{theorem:ADE}.

Alternatively, one can apply Theorem~\ref{theorem:Serre-lift} to obtain a del Pezzo
surface $\hat{S}$ of degree $4$ over a field of characteristic $0$
such that $G\subset\Aut(\hat{S})$, and get a contradiction
with the classification~\mbox{\cite[\S8.6.4]{Dolgachev}}.
\end{proof}

\begin{corollary}\label{corollary:dP4}
Let $S$ be a del Pezzo surface of degree $4$ over a field
of characteristic $p>0$.
Then every group $G\subset\Aut(S)$ contains a normal abelian subgroup of order
coprime to $p$
and index at most $J\cdot |G_{(p)}|^3$, where $J$ is given by~\eqref{eq:auxiliary-subgroups}.
\end{corollary}

\begin{proof}
We know that $\Aut(S)$ is isomorphic to a subgroup of the Weyl group
$$
\WD\cong\mumu_2^4\rtimes\mathfrak{S}_5.
$$
Therefore, by Lemma~\ref{lemma:auxiliary-subgroups}(ii) the group $G$ contains a normal abelian subgroup of order coprime to $p$
and index at most $J\cdot |G_{(p)}|^3$, unless $p=3$
and~\mbox{$G\cong\mumu_2^4\rtimes(\mumu_5\rtimes\mumu_4)$}. The latter case is impossible
by Lemma~\ref{lemma:dP4-forbidden}.
\end{proof}

Next, we discuss the groups acting on del Pezzo surfaces of degree~$3$ or, in other
words, smooth cubic surfaces.

\begin{lemma}\label{lemma:dP3-forbidden}
Let $S$ be a del Pezzo surface of degree $3$ over a field $\Bbbk$ of characteristic $p>0$.
The following assertions hold.
\begin{itemize}
\item[(i)] If $p\neq 2$, then $\Aut(S)$ does not contain subgroups isomorphic
to $\mumu_2^4$.

\item[(ii)] If $p\ge 5$, then $\Aut(S)$ does not contain subgroups isomorphic
to~\mbox{$\HHH_3\rtimes\SL_2(\mathbf{F}_3)$},
and subgroups of order $192$, $288$, and $576$.

\item[(iii)] If $p\ge 5$, and $G\subset \Aut(S)$ is a subgroup of order $162$, then
$G$ is isomorphic to a subgroup of $ \mumu_3^3\rtimes\mathfrak{S}_4$.
\end{itemize}
\end{lemma}

\begin{proof}
Suppose that $p\neq 2$, and $\Aut(S)$ contains a subgroup $G\cong\mumu_2^4$.
By Theorem~\ref{theorem:Serre-lift}, there exists a
del Pezzo surface of degree $3$ over a field of characteristic $0$ with an action of $G$.
The latter is impossible, see for instance~\mbox{\cite[Table~9.6]{Dolgachev}}.
This proves assertion~(i). Alternatively, one can observe that $G$ acts on
$H^0(S, -K_S)\cong\Bbbk^4$. Since $p\neq 2$, this action can be diagonalized, and
now it is easy to check that any cubic polynomial which is preserved by $G$ up to
a scalar is reducible.

Now suppose that $p\ge 5$, and $G\subset \Aut(S)$ is a group whose order is coprime to~$p$.
Again by Theorem~\ref{theorem:Serre-lift}, there exists a
del Pezzo surface of degree $3$ over a field of characteristic $0$ with an action of $G$,
so that $G$ must be a subgroup of one of the groups listed in~\mbox{\cite[Table~9.6]{Dolgachev}}.
However, we see that the orders of all the groups in this table do not exceed $120$,
with the only exception of the group~\mbox{$F=\mumu_3^3\rtimes\mathfrak{S}_4$} of order $648$.
Since $192$, $288$, and $576$ do not divide~$648$, and~$F$ is not isomorphic to $\HHH_3\rtimes\SL_2(\mathbf{F}_3)$, this proves assertions~(ii) and~(iii).
\end{proof}

\begin{corollary}\label{corollary:dP3}
Let $S$ be a del Pezzo surface of degree $3$ over a field
of characteristic $p>0$.
Then every group $G\subset\Aut(S)$ contains a normal abelian subgroup of order
coprime to $p$
and index at most $J\cdot |G_{(p)}|^3$, where $J$ is given by~\eqref{eq:auxiliary-subgroups}.
\end{corollary}

\begin{proof}
We know that $\Aut(S)$ is isomorphic to a subgroup of the Weyl group
$$
\WE\cong\mumu_2\times\PSU_4(\mathbf{F}_2).
$$
Hence $G$ is isomorphic either to a subgroup of $\PSU_4(\mathbf{F}_2)$, or
to $\mumu_2\times G'$, where~$G'$ is a subgroup of $\PSU_4(\mathbf{F}_2)$.
Therefore, it is enough to prove the assertion in the case when
$G\subset\PSU_4(\mathbf{F}_2)$.
If $p=2$ and $G=\PSU_4(\mathbf{F}_2)$, then
$$
|G|=25920=2^6\cdot 3^4\cdot 5< 2^{18}=|G_{(2)}|^3.
$$
If $p\neq 2$, then $G\neq \PSU_4(\mathbf{F}_2)$ by Lemma~\ref{lemma:dP3-forbidden}(i).

Therefore, we may suppose that $G\subsetneq\PSU_4(\mathbf{F}_2)$.
According to~\mbox{\cite[p.~26]{Atlas}},
this implies that $G$ is either a subgroup of
$\mumu_2^4\rtimes\mathfrak{A}_5$, where
the action of $\mathfrak{S}_5$ on~$\mumu_2^4$ comes from the permutation representation;
or of~$\mathfrak{S}_6$; or a subgroup of
$\HHH_3\rtimes\SL_2(\mathbf{F}_3)$; or a subgroup of
$\mumu_3\rtimes\mathfrak{S}_4$; or a subgroup of
a group of order $576$.
Thus, we know from assertions (iii)--(vii)
of Lemma~\ref{lemma:auxiliary-subgroups} that $G$ contains a normal abelian subgroup of order
coprime to $p$
and index at most $J\cdot |G_{(p)}|^3$, unless $p=5$ and either
$G\cong \HHH_3\rtimes\SL_2(\mathbf{F}_3)$, or~\mbox{$|G|=162$}, or $|G|\in\{192, 288, 576\}$.
According to Lemma~\ref{lemma:dP3-forbidden}(ii), the first and the third
of these cases are impossible, so that we can assume $|G|=162$.
Then by Lemma~\ref{lemma:dP3-forbidden}(iii)
we have that $G$ is isomorphic to a subgroup of
$ \mumu_3^3\rtimes\mathfrak{S}_4.
$
In this case $G$ contains a normal abelian subgroup of order
coprime to $p$ and index at most $J\cdot |G_{(p)}|^3$
by Lemma~\ref{lemma:auxiliary-subgroups}(vi).
\end{proof}

\begin{remark}
Over an algebraically closed field of characteristic $2$ (and actually over the finite
field~$\mathbb{F}_4$) there exists a del Pezzo surface of degree~$3$
with an action of the group $\PSU_4(\mathbf{F}_2)$. This is the Fermat cubic surface, i.e.
the cubic surface defined by equation
$$
x^3+y^3+z^3+t^3=0
$$
in the projective space $\PP^3$ with homogeneous coordinates $x$, $y$, $z$, and $t$.
We refer the reader to \cite[Lemma~5.1]{DD} for more details.
\end{remark}

\begin{lemma}\label{lemma:dP2}
Let $S$ be a del Pezzo surface of degree $2$ over a field
of characteristic $p>0$.
Let $J_p(\PP^2)$ be a constant such that every finite group~\mbox{$F\subset\PGL_3(\Bbbk)$}
contains a normal abelian subgroup of order
coprime to $p$ and index at most~\mbox{$J_p(\PP^2)\cdot |F_{(p)}|^3$},
cf.~\eqref{eq:constant-PGL3-odd}.
Then every group $G\subset\Aut(S)$ contains a normal abelian subgroup of order
coprime to $p$
and index at most $J_p(\PP^2)\cdot |G_{(p)}|^3$.
\end{lemma}

\begin{proof}
The anticanonical linear system $|-K_S|$ provides a (separable) double cover~\mbox{$S\to\PP^2$},
see e.g.~\mbox{\cite[Proposition~3.1]{DM-2}} and~\mbox{\cite[\S2.1]{DM-odd}}.
This gives an isomorphism $\Aut(S)\cong\Gamma\times \mumu_2$, where $\Gamma$ is a finite
subgroup of $\Aut(\PP^2)$.
On the other hand, every subgroup $G\subset\Gamma\times\mumu_2$ is either isomorphic to
a subgroup $F$ of $\Gamma$, or is isomorphic to $F\times\mumu_2$, where $F$ is a subgroup of $\Gamma$.
By assumption, $F$
contains a normal abelian subgroup of order
coprime to $p$ and index
at most~\mbox{$J_p(\PP^2)\cdot |F_{(p)}|^3$}.
This immediately implies the required assertion.
\end{proof}

To work with the automorphism groups of del Pezzo surfaces of degree~$1$,
we recall a well-known classification of automorphism groups
of elliptic curve. It can be found in
\cite[Corollary~III.10.2]{Silverman}, \cite[Proposition~A.1.2(c)]{Silverman},
and~\mbox{\cite[Exercise A.1]{Silverman}}.

\begin{theorem}
\label{theorem:elliptic}
Let $C$ be an elliptic curve over a field
of characteristic $p>0$, and let $P\in C$ be a point. Then the stabilizer
of $P$ in $\Aut(C)$ is a subgroup of
\begin{equation}\label{eq:dP1-groups}
\begin{cases}
\mumu_6 \text{\ or\ } \mumu_4, &\text{if\ } p\ge 5,\\
\mumu_3\rtimes\mumu_4, &\text{if\ } p=3.
\end{cases}
\end{equation}
Here in the case $p=3$ the action of $\mumu_4$ on $\mumu_3$ is the (unique) non-trivial one.
\end{theorem}

We will need the following assertion on automorphisms of the last finite group
from Theorem~\ref{theorem:elliptic}.

\begin{lemma}[{see e.g. \cite{GroupNames}}]
\label{lemma:Aut-SL23}
If $\mumu_3\rtimes\mumu_4$ is the non-trivial semi-direct product, then
$$
\Aut(\mumu_3\rtimes\mumu_4)\cong\DD_{12}.
$$
\end{lemma}

\begin{proof}
Straightforward.
\end{proof}

Now we deal with the case of del Pezzo surfaces of degree $1$.

\begin{lemma}\label{lemma:dP1}
Let $S$ be a del Pezzo surface of degree $1$ over a field
of odd characteristic $p$.
Then every group $G\subset\Aut(S)$ contains a normal abelian subgroup of order
coprime to $p$ and index at most
$2J_p(\mathbb{P}^1)\cdot |G_{(p)}|^3$, where
$J_p(\PP^1)$ is given by~\eqref{eq:constant-PGL2}.
\end{lemma}

\begin{proof}
The anticanonical linear system $|-K_S|$ has a unique base point~\mbox{$P\in S$}.
In particular, $P$ is $G$-invariant (and defined over the field of definition of~$S$).
Let $\tilde{S}\to S$ be the blow up of $P$.
Then $G$ acts on $\tilde{S}$, and there is a
$G$-equivariant fibration $\pi\colon \tilde{S}\to\PP^1$ whose
fibers are isomorphic to anticanonical divisors on~$S$.
Observe that $\pi$ cannot be a quasi-elliptic fibration, even if $p=3$,
see~\mbox{\cite[Lemma~3.3(3)]{BLRT}}.
Thus, $\pi$ is an elliptic fibration.
Therefore, the group $G$ fits into an exact sequence
$$
1\to H\to G\to \bar{G}\to 1,
$$
where $\bar{G}\subset\Aut(\PP^1)$, and $H$ acts on the scheme-theoretic generic fiber
$C$ of~$\pi$.
By Lemma~\ref{lemma:P1} the group $\bar{G}$ contains a normal cyclic subgroup $\bar{A}$ of order
coprime to $p$ and index at most $J_p(\PP^1)\cdot |\bar{G}_{(p)}|^3$, where $J_p(\PP^1)$ is given by~\eqref{eq:constant-PGL2}. Let $\alpha$ be some preimage of a generator of $\bar{A}$ in $G$.
Since the order of $\bar{A}$ is coprime to~$p$, we may replace $\alpha$ by its appropriate
power and assume that its order is coprime to $p$ as well.
Note that $\alpha$ acts by an automorphism on the group~$H$.
Furthermore, $H$ fixes the point on $C$ which corresponds
to the $G$-invariant section of $\pi$
contracted to the point $P$. Therefore,
$H$ is a subgroup of one of the groups~\eqref{eq:dP1-groups}
by Theorem~\ref{theorem:elliptic}.

Suppose that $p\ge 5$. Then $H$ is a subgroup of $\mumu_6$ or $\mumu_4$.
In particular, $H$ is a cyclic group, and $|\Aut(H)|\le 2$. Hence $\alpha^2$
commutes with all the elements of~$H$. Let $A$ be the group generated
by $\alpha^2$ and $H$. Then $A$ is abelian. Furthermore, the group generated by the image of $\alpha^2$ is characteristic
in $\bar{A}$, because $\bar{A}$ is cyclic; thus, this group is normal in $\bar{G}$.
This implies that $A$ is normal in $G$. The order of $A$ is coprime to~$p$, and
its index in $G$ is at most
$$
2\cdot\frac{|\bar{G}|}{|\bar{A}|}\le
2J_p(\PP^1)\cdot |\bar{G}_{(p)}|^3=2J_p(\PP^1)\cdot |G_{(p)}|^3.
$$

Now suppose that $p=3$. Then $H\subset \mumu_3\rtimes\mumu_4$.
If $H$ does not contain~\mbox{$\mumu_3$}, then~\mbox{$H\subset\mumu_4$}, and so
$|\Aut(H)|\le 2$. Hence $\alpha^2$
commutes with all the elements of~$H$. In this case the group $A$
generated by $\alpha^2$ and $H$ is abelian,
normal in $G$, has order coprime to~$3$, and its index in $G$
is at most
$$
2\cdot\frac{|\bar{G}|}{|\bar{A}|}\le
2J_3(\PP^1)\cdot |\bar{G}_{(3)}|^3=2J_3(\PP^1)\cdot |G_{(3)}|^3.
$$
Thus, we can suppose that $H\supset\mumu_3$.
Then $H$ is one of the groups $\mumu_3$, $\mumu_6$, or~\mbox{$\mumu_3\rtimes\mumu_4$}.
In the former two cases, one has $\Aut(H)\cong\mumu_2$; in the latter case,
one has $\Aut(H)\cong\DD_{12}$ by Lemma~\ref{lemma:Aut-SL23}.
Observe that the maximal order of an element of~$\DD_{12}$ which is coprime to~$3$ equals~$2$.
Therefore, in any case~$\alpha^2$ commutes with all the elements of $H$.

Let $N$ be the group
generated by~$\alpha^2$ and $H$; note that the order of $N$ is not
coprime to~$3$, and $N$ is not abelian if $H\cong\mumu_3\rtimes\mumu_4$.
The image $\bar{N}$ of~$N$ in $\bar{G}$ is characteristic in $\bar{A}$
and thus normal in $\bar{G}$. Hence $N$ is normal in $G$.

Denote by $H'\subset H$ the trivial subgroup if $H\cong\mumu_3$,
the subgroup isomorphic to $\mumu_2$ if $H\cong\mumu_6$,
and a subgroup isomorphic to $\mumu_4$ if $H\cong\mumu_3\rtimes\mumu_4$.
Let $A\subset N$ be the subgroup generated by $\alpha^2$ and $H'$.
Then $A$ is abelian.
Its index in $N$ does not exceed
$$
\frac{|H|}{|H'|}=3=3\cdot |N_{(3)}|^0.
$$
Therefore, it follows from Corollary~\ref{corollary:Chermak-Delgado} that
$N$ contains a characteristic abelian subgroup $A'$ of order coprime to $3$
and index at most
$$
3^2\cdot |N_{(3)}|^{2\cdot 0+1}=9\cdot |N_{(3)}|
=9\cdot |H_{(3)}|=|H_{(3)}|^3.
$$
Since $A'$ is characteristic in $N$, it is normal in $G$.
Finally, we observe that the index of~$A'$ in~$G$ equals
$$
\frac{|G|}{|N|}\cdot \frac{|N|}{|A'|}=
\frac{|\bar{G}|}{|\bar{N}|}\cdot \frac{|N|}{|A'|}\le
2\cdot \frac{|\bar{G}|}{|\bar{A}|}\cdot \frac{|N|}{|A'|}\le
2\cdot J_3(\PP^1)\cdot |\bar{G}_{(3)}|^3\cdot |H_{(3)}|^3=
2J_3(\PP^1)\cdot |G_{(3)}|^3.
$$
\end{proof}

Next, we summarize the bounds available for del Pezzo surfaces over fields of arbitrary
positive characteristic.

\begin{proposition}\label{proposition:dP}
Let $S$ be a del Pezzo surface of degree $d\not\in\{1, 2, 9\}$ over a field
of characteristic~\mbox{$p>0$}.
Then every group $G\subset\Aut(S)$ contains a normal abelian subgroup of order coprime to $p$
and index at most
$J\cdot |G_{(p)}|^3$, where
\begin{equation}\label{eq:J-dP}
J=\begin{cases}
7200, &\text{if\ } p\ge 7,\\
144, &\text{if\ } p=5,\\
10, &\text{if\ } p=3,\\
3, &\text{if\ } p=2.
\end{cases}
\end{equation}
\end{proposition}

\begin{proof}
It is well known that $1\le d\le 9$.

If $S\cong\PP^1\times\PP^1$, then $G$
contains a normal abelian subgroup of order coprime to $p$
and index at most
$J_p(\PP^1\times\PP^1)\cdot |G_{(p)}|^3$, where~\mbox{$J_p(\PP^1\times\PP^1)$} is given by~\eqref{eq:J-P1xP1}.

If $d=8$ and $S\not\cong\PP^1\times\PP^1$,
then there is an $\Aut(S)$-equivariant conic bundle structure on $S$.
Thus by Proposition~\ref{proposition:CB} the group
$G$ contains a normal abelian subgroup of order coprime to $p$ and index at most
$J_p^{\mathrm{cb}}\cdot |G_{(p)}|^3$, where $J_p^{\mathrm{cb}}$ is given by~\eqref{eq:J-CB}.

If $d=7$, then
there is an $\Aut(S)$-equivariant birational morphism~\mbox{$S\to\PP^1\times\PP^1$}.
This implies that the bound for $\PP^1\times\PP^1$
applies to this case as well.

If $d=6$, then by Lemma~\ref{lemma:dP6} the group
$G$ contains a normal abelian subgroup of order coprime to $p$
and index at most
$J\cdot |G_{(p)}|^3$, where $J$ is given by~\eqref{eq:dP6}.

If $d=5$, $d=4$, or $d=3$, then
$G$ contains a normal abelian subgroup of order
coprime to $p$ and index at most $J\cdot |G_{(p)}|^3$, where $J$ is given by~\eqref{eq:auxiliary-subgroups}, see Lemma~\ref{lemma:dP5} and
Corollaries~\ref{corollary:dP4} and~\ref{corollary:dP3}.

Taking the maximum of the above constants, we obtain
the bound~\eqref{eq:J-dP}.
\end{proof}

Finally, we prove the main result of this section.

\begin{proposition}\label{proposition:dP-odd}
Let $S$ be a del Pezzo surface over a field
of odd characteristic~$p$.
Then every group $G\subset\Aut(S)$ contains a normal abelian subgroup of order coprime to $p$
and index at most
$J_p^{\mathrm{dP}}\cdot |G_{(p)}|^3$, where
\begin{equation}\label{eq:J-dP-odd}
J_p^{\mathrm{dP}}=
\begin{cases}
7200, &\text{if\ } p\ge 7,\\
168, &\text{if\ } p=5,\\
10, &\text{if\ } p=3.
\end{cases}
\end{equation}
\end{proposition}

\begin{proof}
Let $d=K_S^2$ be the degree of $S$.
If $d\not\in\{1,2,9\}$, then the assertion is given by Proposition~\ref{proposition:dP}.
If $d=9$, then $S\cong \PP^2$, and
the assertion is given by Corollary~\ref{corollary:P2}.
If $d=2$, then the same bound as in the case of~$\PP^2$ applies
by Lemma~\ref{lemma:dP2}.
Finally, if $d=1$, then the assertion holds by Lemma~\ref{lemma:dP1}.
\end{proof}

\section{Proof of the main result}
\label{section:proof}

In this section we prove Theorem~\ref{theorem:main}, which is the main result of the paper.

\begin{proof}[Proof of Theorem~\ref{theorem:main}]
Let $G$ be a finite subgroup of $\Cr_2(\Bbbk)$.
Regularizing the birational action of~$G$ on $\PP^2$ and
taking a $G$-equivariant resolution
of singularities, we obtain a smooth projective rational surface $S'$ over $\Bbbk$
with an action of~$G$, see e.g. \cite[Lemma~14.1.1]{P:G-MMP} or~\mbox{\cite[Lemma~3.6]{ChenShramov}}.
Running a $G$-Minimal Model Program on $S'$, we arrive to a rational $G$-minimal surface $S$.
Thus, $S$ is either a del Pezzo surface, or it has a structure of a $G$-equivariant conic bundle, see~\mbox{\cite[Theorem~1G]{Iskovskikh80}}.
In the former case we know from Proposition~\ref{proposition:dP-odd}
that $G$ contains a normal abelian subgroup of order coprime to $p$
and index at most
$J_p^{\mathrm{dP}}\cdot |G_{(p)}|^3$, where $J_p^{\mathrm{dP}}$ is given by~\eqref{eq:J-dP-odd}.
In the latter case we know from Proposition~\ref{proposition:CB}
that $G$ contains a normal abelian subgroup of order coprime to $p$
and index at most
$J_p^{\mathrm{cb}}\cdot |G_{(p)}|^3$, where $J_p^{\mathrm{cb}}$ is given by~\eqref{eq:J-CB}.
Therefore, $G$ always contains a normal abelian subgroup of order coprime to $p$
and index at most~\mbox{$J_p(\Cr_2)\cdot |G_{(p)}|^{3}$},
where
$$
J_p(\Cr_2)=\max\{J_p^{\mathrm{dP}}, J_p^{\mathrm{cb}}\}=J_p^{\mathrm{dP}}=
\begin{cases}
7200, &\text{if\ } p\ge 7,\\
168, &\text{if\ } p=5,\\
10, &\text{if\ } p=3.
\end{cases}
$$

Now let us show that this bound is sharp.
Let $\Bbbk$ be an algebraically closed field of characteristic $p>0$.
Then the group $\Cr_2(\Bbbk)$ contains a subgroup $\PGL_2(\mathbf{F}_{p^k})$
for any $k\ge 1$. Therefore, by Example~\ref{example:PGL}
one cannot choose a constant $J$ and a constant $e<3$
such that for any finite subgroup
$G\subset \Cr_2(\Bbbk)$ there exists a normal abelian
subgroup in $G$ whose index is at most $J\cdot |G_{(p)}|^e$.

Furthermore, if $p\ge 7$, note that
$\Cr_2(\Bbbk)\supset\Aut(\PP^1\times\PP^1)$, and the latter group contains
a subgroup $G=(\mathfrak{A}_5\times\mathfrak{A}_5)\rtimes\mumu_2$,
where the non-trivial element of $\mumu_2$ acts by interchanging the factors.
The group $G$ does not contain
non-trivial normal abelian subgroups. In other words, it does not contain normal
abelian subgroups of index less than
$$
|G|=7200=7200\cdot |G_{(p)}|^3.
$$

If $p=5$, observe that a simple non-abelian
group $G=\PSL_2(\mathbf{F}_7)$ acts on $\PP^2$; the action can be constructed using the same formulas as in characteristic~$0$, see for instance~\mbox{\cite[\S6.5.3]{Dolgachev}}.
Its only normal abelian subgroup is trivial; in other words, $G$ does not contain
normal abelian subgroups of index less than
$$
|G|=168=168\cdot |G_{(5)}|^3.
$$

Finally, if $p=3$, consider the dihedral group $\bar{G}=\DD_{10}$ acting on $\PP^1$.
Let~\mbox{$\Sigma\subset\PP^1$} be a $\bar{G}$-orbit of length $5$. Let
$v_2\colon \PP^1\hookrightarrow\PP^2$ be the Veronese embedding,
and let~\mbox{$S\to \PP^2$} be the blow up of the five points
of $v_2(\Sigma)$. Then $S$ is a del Pezzo surface of degree $4$. Observe that $\Aut(S)$ contains
the group
$$
G\cong\mumu_2^4\rtimes \bar{G},
$$
where the action of $\bar{G}$ on $\mumu_2^4$ comes from
the permutation representation (see e.g.~\mbox{\cite[\S3]{DD}}).
Thus, one has $G\subset\Cr_2(\Bbbk)$.
It is easy to see that $G$ does not contain normal abelian subgroups of
order less than
$$
|\bar{G}|=10=10\cdot |G_{(3)}|^3.
$$
\end{proof}

\begin{remark}\label{remark:future}
If $\Bbbk$ is a field of characteristic $2$, we are unable to prove
an analog of Theorem~\ref{theorem:main} yet, mostly due to our lack of
good understanding of finite subgroups in $\PGL_3(\Bbbk)$. For instance, the partial
classification of such groups provided in~\cite{Hartley} or~\mbox{\cite[Theorem~2.7]{King}}
is rather similar to Theorem~\ref{theorem:Mitchell} but lists only maximal finite groups,
which requires further group-theoretic work to obtain the estimates
(cf. Examples~\ref{example:constant-subgroup-1}
and~\ref{example:constant-subgroup-2}).
Note however that our bounds for indices of normal subgroups
in automorphism groups of conic bundles
(see Proposition~\ref{proposition:CB}) and
del Pezzo surfaces of degree~\mbox{$3\le d\le 8$}
(see Proposition~\ref{proposition:dP})
include the case of characteristic~$2$.
Based on these bounds and other preliminary computations,
we believe that over a field $\Bbbk$ of characteristic $2$
every finite subgroup~$G$ of~\mbox{$\Cr_2(\Bbbk)$}
contains a normal abelian subgroup of order coprime
to $2$ and index at most~\mbox{$3\cdot |G_{(2)}|^{3}$}.
Also, one can see that there does not exist a stronger bound
when $\Bbbk$ is algebraically closed.
Indeed, the group $G=\mumu_7\rtimes\mumu_3$ acts on $\PP^2$
and does not contain normal abelian subgroups of index
less than $3=3\cdot |G_{(2)}|^3$.
\end{remark}

\end{document}